\documentclass[11pt]{amsart}

\usepackage{amsmath}%
\usepackage{amsfonts}%
\usepackage{amssymb}%
\usepackage{color}
\usepackage{graphicx}
\usepackage{setspace}
\usepackage{enumerate}
\usepackage{mathrsfs}
\usepackage{hyperref}
\usepackage{float}
\allowdisplaybreaks

\title[The TST in Carnot Groups]{The Traveling Salesman Theorem in Carnot groups}
\author{Vasilis Chousionis, Sean Li, and Scott Zimmerman}

\address{V.\ Chousionis: Department of Mathematics, University of Connecticut, 
	341 Mansfield Road U1009, Storrs, Connecticut 06269, USA, {\tt vasileios.chousionis@uconn.edu}}

\address{S.\ Li: Department of Mathematics, University of Connecticut, 
	341 Mansfield Road U1009, Storrs, Connecticut 06269, USA, {\tt sean.li@uconn.edu}}

\address{S.\ Zimmerman: Department of Mathematics, University of Connecticut, 
	341 Mansfield Road U1009, Storrs, Connecticut 06269, USA, {\tt scott.zimmerman@uconn.edu}}

\thanks{V.C. is supported by  the Simons Collaboration grant no.\  521845. S.L. is supported by NSF grant DMS-1600804.}
plus 6pt minus 12pt
\makeatletter
\newtheorem*{rep@theorem}{\rep@title}
\newcommand{\newreptheorem}[2]{%
\newenvironment{rep#1}[1]{%
 \def\rep@title{#2 \ref{##1}}%
 \begin{rep@theorem}}%
 {\end{rep@theorem}}}
\makeatother

\newtheorem{theorem}{Theorem}
\newreptheorem{theorem}{Theorem}
\newtheorem{lemma}[theorem]{Lemma}
\newtheorem{corollary}[theorem]{Corollary}
\newtheorem{proposition}[theorem]{Proposition}

\def\diam{{\rm diam\,}}

\theoremstyle{definition}

\newcommand{\barint}{
\rule[.036in]{.12in}{.009in}\kern-.16in \displaystyle\int }

\newcommand{\barcal}{\mbox{$ \rule[.036in]{.11in}{.007in}\kern-.128in\int $}}
\newcommand{\G}{\mathbb G}
\newcommand{\ra}{\rightarrow}
\newcommand{\ha}{\mathcal{H}}
\newcommand{\stm}{\setminus}
\newcommand{\ve}{\varepsilon}
\newcommand{\Z}{\mathbb{Z}}
\newcommand{\Heis}{\mathbb H}
\newcommand{\R}{\mathbb R}

\def\diam{\operatorname{diam}}

\def\mvint_#1{\mathchoice
          {\mathop{\vrule width 6pt height 3 pt depth -2.5pt
                  \kern -8pt \intop}\nolimits_{\kern -3pt #1}}%
          {\mathop{\vrule width 5pt height 3 pt depth -2.6pt
                  \kern -6pt \intop}\nolimits_{#1}}%
          {\mathop{\vrule width 5pt height 3 pt depth -2.6pt
                  \kern -6pt \intop}\nolimits_{#1}}%
          {\mathop{\vrule width 5pt height 3 pt depth -2.6pt
                  \kern -6pt \intop}\nolimits_{#1}}}

\numberwithin{theorem}{section} \numberwithin{equation}{section}

\begin{document}
\begin{abstract}
Let $\G$ be any Carnot group. We prove that, if a subset of $\G$  is contained in a rectifiable curve, then it satisfies Peter Jones' geometric lemma with some natural modifications. We thus prove one direction of the Traveling Salesman Theorem in $\G$. Our proof depends on new Alexandrov-type curvature inequalities for the Hebisch-Sikora metrics. We also apply the geometric lemma to prove that, in every Carnot group, there exist $-1$-homogeneous Calder\'on-Zygmund kernels  such that, if a set $E \subset \G$ is contained in a 1-regular curve, then the corresponding singular integral operators are bounded in $L^2(E)$. In contrast to the Euclidean setting, these kernels are nonnegative and symmetric.
\end{abstract}

\maketitle

\section{Introduction}

Let $X$ be a metric space. A set $\Gamma \subset X$ is called a {\em rectfiable curve} if it is the Lipschitz image of a finite interval. The Analyst's Traveling Salesman Problem asks the following: given a set $E \subset X$, is there a finite length rectifiable curve $\Gamma \subset X$ so that $E \subseteq \Gamma$?  This would mean that it is possible to visit the set $E$ in finite time.  In the case when such curves $\Gamma$ exist, 
one can also ask for the smallest length of $\Gamma$. %Recall that a set is called  {\em rectfiable curve} if it is the Lipschitz image of a finite interval.

When $X = \R^2$, Jones gave a complete answer to the first question using the notion of {\em $\beta$-numbers} \cite{JonesTSP}.  For $E \subset X$,  $x \in \R^2$, and $r > 0$ we define
\begin{align*}
  \beta_E(x,r) := \inf_L \sup_{z \in B(x,r) \cap E} \frac{d(z,L)}{r},
\end{align*}
where the infimum is taken over the set of all affine lines $L$.  Thus, $\beta_E(x,r)$ is a scale-invariant measure of how close the set $E$ lies to some line.  He also developed upper and lower bounds for the infimal length of rectifiable curves containing $E$ by using these $\beta$-numbers.  Okikiolu later generalized his result to Euclidean spaces of all dimensions \cite{Okikiolu}.  The following theorem is now known as the Traveling Salesman Theorem:
\begin{theorem}[Euclidean Traveling Salesman Theorem (TST) \cite{JonesTSP,Okikiolu}]
\label{etst}
  Let $E \subset \R^n$.  Then $E$ lies on a finite length rectifiable curve if and only if %we have that
  \begin{align} \label{carleson}
    \gamma(E) := \diam(E) + \int_{\R^n} \int_0^\infty \beta_E(x,r)^2 ~\frac{dr}{r^n} dx < \infty.
  \end{align}
  Furthermore, if $\gamma(E) < \infty$, then we have an estimate on the infimal length of such curves:
  \begin{align*}
    \frac{1}{C} \gamma(E) \leq \inf_{\Gamma \supset E} \ell(\Gamma) \leq C \gamma(E),
  \end{align*}
  where $C$ is some constant depending only on $n$.
\end{theorem}
It is well known from Rademacher's theorem that rectifiable curves in $\mathbb{R}^n$ infinitesimally resemble lines. 
However, to answer questions about the boundedness of singular integrals and other problems of a global nature, 
Rademacher's theorem does not provide enough quantitative information. 
Stated informally, one would like to know that rectifiable curves admit good affine approximations ``at most places and scales". 
This is typically quantified via integrals over space and scale as in \eqref{carleson}. 
Such {\em Carleson integrals} convey the right quantitative information required for the study of certain well known singular integrals. 
Jones was the first to realize this connection \cite{jo1}. 
He used $\beta$-numbers to control the Cauchy singular integral on $1$-dimensional Lipschitz graphs. 
Since Jones' pioneering work, $\beta$-numbers have become crucial tools in harmonic analysis, geometric measure theory, and their connections. 
In fact, the introduction of $\beta$-numbers may be viewed as a point of departure for the theory of {\em quantitative rectifiability} which was developed in the 90's by David and Semmes \cite{DS1,DS2, DS3}. 
The study of quantitative rectifiability  led to a rich geometric framework for singular integrals acting on lower dimensional subsets of $\R^n$. For more information, we refer the reader to the books \cite{DS2,pajotbook,tolsabook}

There have been numerous generalizations and variants of Theorem~\ref{etst} beyond Euclidean spaces. Schul \cite{SchulTSP} extended Theorem \ref{etst} to Hilbert spaces, David and Schul \cite{DavidSchul} recently considered the theorem in the graph inverse limits of Cheeger-Kleiner, and Hahlomaa and Schul (independently)\cite{Hah05, Hah07,Sch07} obtained variants of Theorem \ref{etst} in general metric spaces. In the last case, however, there is no natural notion of lines over which one may infimize in the definition of $\beta$,
so curvature-type quantities other than $\beta$-numbers must be considered. 

A natural class of metric spaces in which to study the Analyst's Traveling Salesman Problem (TSP) is the class of Carnot groups
(introduced in more detail in Section~\ref{CarnotSec}).  
This is a special subclass of nilpotent Lie groups whose abelian members are precisely the Euclidean spaces.  
Thus, these groups can be viewed as nonabelian generalizations of Euclidean spaces. Moreover, Carnot groups are locally compact geodesic spaces which admit dilations, and they are isometrically homogeneous. In fact, by a recent observation of Le Donne \cite{ledo}, Carnot groups are  the {\em only} metric spaces with these properties. 
Developing aspects of quantitative rectifiability (such as the TST) in Carnot groups contributes to the systematic effort which started about 15 years
ago to develop Geometric Measure Theory (GMT) on these sub-Riemannian spaces. Rather than providing a long list of highlights in sub-Riemannian GMT, 
we refer the reader to the recent lecture notes of Serra Cassano \cite{scnotes} which provide a nice overview of the field with ample references to the continuously growing literature. 

Like Euclidean spaces, Carnot groups are Ahlfors regular and contain a rich family of lines 
(which are cosets of 1-dimensional subgroups isometric to $\R$).
These are the so-called {\em horizontal lines}.
Hence the definition of $\beta$-numbers readily generalizes in this case.  
Indeed, in the definition of $\beta_E(x,r)$, we instead take the infimum $\inf_L$ 
over all horizontal lines that intersect $B(x,r)$, 
and we use the sub-Riemannian metric to measure distance. Ferrari, Franchi and Pajot \cite{FFP} initialized the study of the TSP in the simplest nonabelian Carnot group; the Heisenberg group $\Heis$. They proved that, if the Carleson integral of $\beta_E^{2}$ is bounded, then $E$ lies on a rectifiable curve. Schul and the second named  author \cite{LiSchul,LiSchul2}  improved the aforementioned result, and they obtained an almost sharp Traveling Salesman Theorem in $\Heis$:
\begin{theorem}[\cite{LiSchul}] \label{t:H-TSP1}
  There exists a universal constant $C > 0$ so that if $\Gamma \subset \Heis$ is a finite length rectifiable curve, then
  \begin{align*}
    \diam(\Gamma) + \int_\Heis \int_0^\infty \beta_\Gamma(x,r)^4 ~\frac{dr}{r^4} ~dx \leq C \ell(\Gamma).
  \end{align*}
\end{theorem}

\begin{theorem}[\cite{LiSchul2}] \label{t:H-TSP2}
  For any $p < 4$, there exists $C(p) > 0$ so that for any $E \subset \Heis$ for which
  \begin{align*}
    \gamma_p(E) := \diam(E) + \int_\Heis \int_0^\infty \beta_E(x,r)^p ~\frac{dr}{r^4} ~dx < \infty,
  \end{align*}
  then there is a finite length rectifiable curve $\Gamma$ that contains $E$ and
  \begin{align*}
    \ell(\Gamma) \leq C_p \gamma_p(E).
  \end{align*}
\end{theorem}
It is currently unknown whether Theorem \ref{t:H-TSP2} holds for $p = 4$. This would give a sharp converse to Theorem \ref{t:H-TSP1}.  Note that the 4 in the exponent of $\frac{dr}{r^4}$ is an obvious modification resulting from the Hausdorff dimension of the Heisenberg group. However, the exponent 4 of $\beta_E$ in Theorem \ref{t:H-TSP1} is a consequence of an Alexandrov-type curvature inequality in $\Heis$ whose rather delicate proof depends crucially upon the Koran\'yi metric in $\Heis$. Note, however, that Theorem~\ref{t:H-TSP1} holds for any homogeneous metric in $\Heis$ including the sub-Riemannian metric.

We cannot use the Koran\'yi metric in the general setting since it does not generalize to arbitrary Carnot groups.  Instead, we use another family of metrics -- the Hebisch-Sikora metrics \cite{HebSik} -- for which we will establish a similar curvature inequality (Theorem~\ref{Goal}).  With the new curvature inequality, we can then use the proof of \cite{LiSchul} to obtain the following theorem  which holds for all homogenous metrics in any Carnot group $\G$.

%In the present paper we will  generalize  Theorem \ref{t:H-TSP1} to {\em any} Carnot group.
%A natural question to ask is the following:
%do either of these directions of the Traveling Salesman Theorem hold in the class of all Carnot groups?  
%One of the main ingredients in the proof of Theorem~\ref{t:H-TSP1} is an Alexandrov curvature inequality which depends specifically upon the Koranyi metric in $\Heis$ for which there is no easy generalization to an arbitrary Carnot group.
%(We do note, however, that Theorem~\ref{t:H-TSP1} holds for any homogeneous metric which is bi-Lipschitz equivalent to the Koranyi metric including the sub-Riemannian metric).  
%We will instead prove a similar curvature inequality (Theorem~\ref{Goal}) for another family of metrics: the Hebisch-Sikora metrics \cite{HebSik}. 
%These metrics will be defined in Section~\ref{CarnotSec} and exist in any Carnot group.  This curvature inequality will give the following theorem:

\begin{theorem}
\label{TSP}
Let $\mathbb{G}$ be a step $r$ Carnot group with Hausdorff dimension $Q$.  There is a constant $C=C(\mathbb{G})>0$
such that, for any rectifiable curve $\Gamma \subset \mathbb{G}$, we have
$$
\int_{\mathbb{G}} \int_0^{\infty} \beta_{\Gamma}(B(x,t))^{2r^2} \frac{dt}{t^Q} d \mathcal{H}^Q(x) \leq C \mathcal{H}^1(\Gamma).
$$
\end{theorem}
In the case of step 2 Carnot groups (of which the Heisenberg group is an example), 
this theorem provides a bound on the Carleson integral involving $\beta^{2 \cdot 2^2} = \beta^8$.
This is weaker than the bound on the Carleson integral of Theorem~\ref{t:H-TSP1} which involves $\beta^4$.  
We will prove in Section~\ref{Step2Sec} that, in the special case of step 2 Carnot groups, the curvature inequality can be improved so that Theorem~\ref{TSP} holds with an exponent 4 on $\beta$. Therefore we obtain a genuine generalization of Theorem~\ref{t:H-TSP1} to any step 2 Carnot group.
\begin{theorem}
\label{TSP2}
Let $\mathbb{G}$ be a step 2 Carnot group with Hausdorff dimension $Q$.  There is a constant $C=C(\mathbb{G})>0$
such that, for any rectifiable curve $\Gamma \subset \mathbb{G}$, we have
$$
\int_{\mathbb{G}} \int_0^{\infty} \beta_{\Gamma}(B(x,t))^{4} \frac{dt}{t^Q} d \mathcal{H}^Q(x) \leq C \mathcal{H}^1(\Gamma).
$$
\end{theorem}

As mentioned earlier, there are deep connections between quantitative rectifiability and singular integral operators (SIO) in Euclidean spaces. In particular, the boundedness of SIOs on Lipschitz graphs (and beyond) is a classical topic developed by Calder\'on \cite{Calderon}, Coifman-McIntosh-Meyer \cite{CMM}, David \cite{MR956767}, David-Semmes \cite{DS1, DS2}, Tolsa \cite{tolsaplms}, and many others. In all of these contributions, the kernels defining the SIO are odd functions. This is very reasonable since, in order to define a SIO which makes sense on lines and other ``nice'' $1$-dimensional objects, one heavily relies on the cancellation properties of the kernel, see e.g. \cite[Proposition 1, pp 289]{stein}. Surprisingly, the situation is very different in Carnot groups, and this was first observed in the first Heiseinberg group in \cite{ChoLi}. Using Theorem \ref{TSP}, we will prove the following theorem.

\begin{theorem}
\label{siosintro}
Let $(\G,d)$ be Carnot group of step $r\geq 2$ equipped with a homogeneous metric $d$. There exists a nonnegative, symmetric, $-1$ homogeneous, Calder\'on-Zygmund kernel  $K : \G \stm \{0\} \ra (0,\infty)$ 
such that  the corresponding truncated singular integrals
$$T^\ve f\,(p)=\int_{E \stm B_{d}(p,\ve)} K(q^{-1} \cdot p) f(q) \,d \ha^1(q)$$
are uniformly bounded in $L^2(\ha^1 |_E)$ for every $1$-regular set $E$  which is contained in a $1$-regular curve.
\end{theorem}

The paper is organized as follows.
In Section~\ref{CarnotSec}, we will introduce the basic properties of Carnot groups that will be needed for our purposes,
and we will define the Hebisch-Sikora metric used throughout the paper.
We will introduce and prove the curvature estimate Theorem~\ref{Goal} in Section~\ref{curvatureSec}.
This curvature bound will be used to prove Theorem~\ref{TSP} in Section~\ref{TSTSec}.
This section follows the example set forth in \cite{LiSchul}.
The case of step 2 groups will be addressed in Section~\ref{Step2Sec}. Finally in Section~\ref{SIOSec} we will prove Theorem \ref{siosintro}.
%and an application of Theorem~\ref{TSP} to the theory of singular integral operators
%s contained in Section~\ref{SIOSec}.

\section{Carnot preliminaries}
\label{CarnotSec}

A step $r$ \emph{Carnot group} is a connected, simply connected Lie group $\mathbb{G}$ 
whose Lie algebra $\mathfrak{g}$ is \emph{stratified} in the following sense:
$$
\mathfrak{g} = V_1 \oplus \cdots \oplus V_r, 
\quad [V_1,V_i] = V_{i+1}  \text{ for } i=1,\dots,r-1,
\quad [V_1,V_r] =\{0\}
$$
where $V_1,\dots,V_r$ are non-zero subspaces of the Lie algebra.
Any such Lie group may be identified with $\mathbb{R}^N$ for some $N \in \mathbb{N}$ via the exponential coordinates on $\mathfrak{g}$.
Denote by $|\cdot|$ the Euclidean norm in $\mathbb{G} = \mathbb{R}^N$.
Say $Q$ is the \emph{homogeneous dimension} of $\mathbb{G}$ 
i.e. $Q = \sum_{i=1}^r i \dim V_i$.
There is a natural family of automorphisms known as \emph{dilations} on $\mathbb{G}$.
If, for any $p \in \mathbb{G}$, we write $p = (p_1,\dots,p_r)$ where
$p_i \in \mathbb{R}^{v_i}$ for $v_i = \dim V_i$,
then for any $s>0$ define the dilation
$$
\delta_s(p) = \left( sp_1,s^{2}p_2, \dots, s^{r}p_r \right).
$$
It follows that $\{\delta_s\}_{s>0}$ is a one parameter family
i.e. $\delta_s \circ \delta_t = \delta_{st}$.
Given $p \in \mathbb{G}$, 
we will also often write $p = (p_1,p_2) \in \mathbb{R}^{v_1} \times \mathbb{R}^{N-v_1}$.
We may then think of $p_1$ as the ``horizontal part'' of $p$.
Define the \emph{non-horizontal part} of $p \in \mathbb{G}$ as 
$$
NH(p) := \tilde{\pi}(p)^{-1}p
$$
where $\tilde{\pi}:\mathbb{G} \to \mathbb{G}$ is the map $\tilde{\pi}(p_1,p_2) = (p_1,0)$ (note that this is not a projection!).

We will now endow $\mathbb{G}$ with a metric space structure.
\begin{theorem}[Hebisch and Sikora, 1990]
\label{HSdef}
There exists $\varepsilon_0 > 0$ so that, for every $\eta < \varepsilon_0$,
$$
\Vert x \Vert 
%= \inf \{ t \, : \, \delta_{1/t}(x) \in B_{\mathbb{R}^N}(0,\eta) \} 
= \inf \{ t \, : \, |\delta_{1/t}(x)| < \eta \}
\quad
\text{for all } x \in \mathbb{G}
$$
is a homogeneous, subadditive norm on $\mathbb{G}$
i.e. for every $s>0$ and $x,y \in \mathbb{G}$, $\Vert \delta_s(x) \Vert = s \Vert x \Vert$
and $\Vert xy \Vert \leq \Vert x \Vert + \Vert y \Vert$.
In particular, the unit ball in $\Vert \cdot \Vert$ coincides with the Euclidean ball $B_{\mathbb{R}^N}(0,\eta)$.
\end{theorem}

For any $\eta < \varepsilon_0$,
call this norm the \emph{Hebisch-Sikora (HS) norm} on $\mathbb{G}$ as introduced in \cite{HebSik}.
Define the induced metric $d$ on $\mathbb{G}$ as
$
d(x,y) = \Vert y^{-1}x \Vert.
$
The continuity of the Carnot dilations implies in particular that $|\delta_{1/\Vert x \Vert}(x)| = \eta$.
For any horizontal point (that is, $p \in \mathbb{R}^n \times \{0 \}$),
we have 
\begin{equation}
\label{etafoot}
\Vert p \Vert 
= \inf \{ t \, : \, \tfrac{1}{t} | p | < \eta \} 
= \inf \{ t \, : \, \tfrac{1}{\eta} | p | < t \}
= \tfrac{1}{\eta} | p |.
\end{equation}
Moreover, we have $\Vert \tilde{\pi} (p) \Vert \leq \Vert p \Vert$ for any $p \in \G$. 
Indeed,
if there was some $t>0$ satisfying both $\Vert \tilde{\pi}(p) \Vert > t$ and $|\delta_{1/t}(p)| < \eta$, %as in the definition of the HS norm,
we would have $\eta^2 > \frac{p_1^2}{t^2} + \cdots + \frac{p_r^2}{t^{2r}} \geq \frac{p_1^2}{t^2} = \frac{\Vert \tilde{\pi}(p) \Vert^2}{t^2} \eta^2 > \eta^2$ 
which is impossible.  We also record that for any compact $K \subset \mathbb{G}$,
there is a constant $C>0$ (depending on $K$) so that
\begin{equation}
\label{compact}
d(x,y)^r \leq C|x-y|
\quad 
\text{for any } x,y \in K.
\end{equation}

 A metric $d$ on $\G$ is said to be {\it homogeneous} if $d: \R^N \times \R^N \ra [0,\infty)$ is
continuous with respect to the Euclidean topology, is left invariant,
and is $1$-homogeneous with respect to the dilations
$\{\delta_r\}_{r>0}$. The $1$-homogeneity of $d$ means that
$$
d(\delta_r(p),\delta_r(q)) = r\, d(p,q)
$$
for all $p,q\in\G$ and all
$r>0$. We note in particular that the Hebisch-Sikora  and the Carnot--Carath\'eodory metrics
are homogeneous.  Any two homogeneous metrics
$d_1$ and $d_2$ on a given Carnot group $\G$ are equivalent in the
sense that there exists a constant $L>0$ so that
\begin{equation}\label{quasiconvexity}
L^{-1}d_1(p,q) \le d_2(p,q) \le Ld_1(p,q)
\end{equation}
for all $p,q\in\G$; this is an easy
consequence of the assumptions.
%If $d_1$ is any other metric on $\mathbb{G}$ induced by a homogeneous norm (or $d_1$ is the traditional Carnot-Carath\'{e}odory metric),
%then, by the compactness of the unit sphere, we have %there is a constant $M$ depending only on $\mathbb{G}$ so that
%$$
%d(x,y) \lesssim d_1(x,y) \lesssim d(x,y) 
%\quad 
%\text{for any } x,y \in \mathbb{G}.
%$$

%Our goal will be to prove Theorem~\ref{TSP}.
We define the Jones $\beta$-numbers for a set $K \subset \mathbb{G}$ as follows: for any $x \in \G$ and $r>0$,
%for a closed ball $B = B(x,r) \subset \mathbb{G}$,
%\footnote{The second equality follows from the fact that $\diam(B(x,t)) = 2 t$ for any left invariant, homogeneous metric in $\mathbb{G}$. \cite[Proposition 2.4]{FraSerSerOnthe}}
$$
\beta_K(B(x,r)) 
%= \inf_L \sup_{x \in K \cap B(x,t)} \frac{d(x,L)}{\diam (B(x,t))}
= \inf_L \sup_{z \in K \cap B(x,r)} \frac{d(z,L)}{r}
$$
where the infimum is taken over all possible \emph{horizontal lines}
$$
L = \{p \delta_s(\tilde{\pi}(p^{-1}q)) \, : \, s \in \mathbb{R}\} 
\quad \text{where } p,q \in \mathbb{G}.
$$

The following is the famous Baker-Campbell-Hausdorff formula.
\begin{theorem}[Dynkin, '47]
\label{BCH}
Suppose $e^{(\cdot)}:\mathfrak{g} \to \mathbb{G}$ is the exponential map.
Given $X,Y \in \mathfrak{g}$, 
choose $Z$ such that $e^Z = e^Xe^Y$.
Then
$$
Z = \sum_{k=1}^{\infty} \frac{(-1)^{k-1}}{k} \sum%_{\substack{r_1+s_1 > 0 \\ \vdots \\ r_n+s_n > 0}}
P(r_1,s_1,\dots,r_k,s_k)
%\frac{1}{\sum_{i=1}^n (r_i+s_i) \prod_{i=1}^n r_i!s_i!}
%\left[ X^{r_1} Y^{s_1} X^{r_2} Y^{s_2} \dots X^{r_n}Y^{s_n}\right]
[\underbrace{X, \cdots [X}_{r_1},[ \underbrace{Y,\cdots [Y}_{s_1}, \cdots
[\underbrace{X, \cdots [X}_{r_k},[ \underbrace{Y,\cdots Y}_{s_k}]]] \cdots]
$$
where the second sum is taken over all $\{r_1,s_1,\dots,r_k,s_k\} \in \mathbb{N}^{2k}$
satisfying $r_i + s_i > 0$ for $i=1,\dots,k$,
and
$$
P(r_1,s_1,\dots,r_k,s_k) = \frac{1}{\sum_{i=1}^k (r_i + s_i) \prod_{i=1}^k r_i! s_i!}.
$$
\end{theorem}
Notice that the nested commutators vanish if $s_k > 1$ or if $s_k = 0$ and $r_k>1$.
Also, since $\mathbb{G}$ is nilpotent, 
the first sum terminates after finitely many terms,
and the length of the nested brackets is bounded from above.
That is, there are only finitely many summed nested bracket terms in the BCH formula for a Carnot group $\mathbb{G}$.

We will later make use of the following estimates established in \cite{HebSik} (for a verification, see the proof of Lemma~\ref{NH}).
Choose $a,b \in \mathbb{G}$ with $|a|<1$ and $|b|<1$.
Write $a=(a_1,a_2)$ and $b=(b_1,b_2)$ as above.
Then
$$
ab = (a_1+b_1,a_2+b_2+R(a,b))
$$
for some polynomial $R$ given by the BCH formula (Theorem~\ref{BCH}).
Write $R_1(a,b) = R((a_1,0),(b_1,0))$ and $R_2 = R-R_1$.
Then the BCH formula gives
\begin{equation}
\label{HS2}
|R_2(a,b)| \leq C_2(|a_1||b_2| + |a_2||b_1| + |a_2||b_2|)
\end{equation}
and
\begin{equation}
\label{HS1}
|R_1(a,b)| \leq C_1|a_1||b_1| \left| \frac{a_1}{|a_1|} - \frac{b_1}{|b_1|} \right|.
\end{equation}
for some constants $C_1$ and $C_2$ depending only on the group structure of $\G$.

For the remainder of the paper, 
fix a positive constant $\eta < \min \{ \varepsilon_0, \tfrac12 \}$ 
(where $\varepsilon_0$ is as in Theorem~\ref{HSdef}) such that,
if $|a| \leq \eta$ and $|b| \leq \eta$, then
\begin{equation}
\label{HS3}
(C_2+1)(|a_1|+|a_2|+|b_1| +|b_2|)\leq \frac18
\end{equation}
and 
\begin{equation}
\label{HS4}
(5C_1^2+1) |a_1||b_1|  \leq \frac12.
\end{equation}

Throughout the paper, we will write $a \lesssim b$
to indicate that there is a constant $C>0$ depending only on the metric space $(\mathbb{G},d)$
satisfying $a \leq Cb$.
Similarly, we will write $a \lesssim_{\xi} b$ if the constant depends also on some other parameter $\xi$.
%Note that, for any $t > 0$, \begin{equation}\label{NHNorm} \left\Vert NH\left( \delta_{t}(p) \right) \right\Vert = d \left( \tilde{\pi} \left( \delta_{t}(p) \right), \delta_{t}(p) \right)=t \, d(\tilde{\pi}(p),p) = t \Vert NH(z) \Vert. \end{equation}

\section{Curvature bound in a Carnot group}
\label{curvatureSec}

For $p,q \in \mathbb{G}$, denote the \emph{horizontal segment} between them as
$$
L_{pq} := \{ p \, \delta_t(\tilde{\pi}( p^{-1} q)) \, : \, t \in [0,1] \}.
$$
While this segment will always originate at $p$,
it will not intersect $q$ in general.
Note also that horizontal segments do not necessarily coincide with 
Euclidean segments if the step of $\mathbb{G}$ is $r > 2$.
For each $t \in [0,1]$, write $L_{pq}(t) = p \, \delta_t(\tilde{\pi}( p^{-1} q))$.
If $p=0$, then $L_q := L_{0q} \subset \mathbb{R}^n \times \{0\}$ is the segment
$$
L_q = \{ \delta_t(\tilde{\pi}(q)) \, : \, t \in [0,1] \} = \{ (tq_1,0) \, : \, t \in [0,1] \}.
$$
%Note that this segment starts at the origin and ends directly ``below'' $q$ in the horizontal layer.
That is, $L_q$ is simply the Euclidean line segment from the origin to $\tilde{\pi}(q) = (q_1,0)$.
Our goal in this section will be to prove the following curvature estimate in $\mathbb{G}$.
Here, fix the value $m = 2^{-217}$.
(This value will be important in Section~\ref{TSTSec}.
The theorem actually holds for any $0 < m < 1$, but then the constant $C_0$ would depend also on $m$.)

\begin{theorem}
\label{Goal}
%Fix $0<m<1$. (For our purposes later, we will use $m = 2^{-217}$.)
Suppose $a,z,v,w \in \mathbb{G}$ satisfy
$$
m \rho \leq \min \{ d(a,z), d(a,v), d(z,v), d(v,w) \}
$$
and
$$
\max \{ d(a,z), d(a,v), d(z,v), d(v,w), d(a,w) \} \leq \rho
$$
for some $\rho > 0$.
Then there is a constant $C_0 = C_0(\mathbb{G}) > 0$ so that
$$
\sup_{t \in [0,1]} d(L_{av}(t),L_{aw})^{2r^2} + \sup_{t \in [0,1]} d(L_{vw}(t),L_{aw})^{2r^2}
\leq C_0 \rho^{2r^2-1} \Delta
$$
where $\Delta := d(a,z) + d(z,v) + d(v,w) - d(a,w)$.
\end{theorem}
\begin{figure}[H]
\centering
\includegraphics[scale = 1]{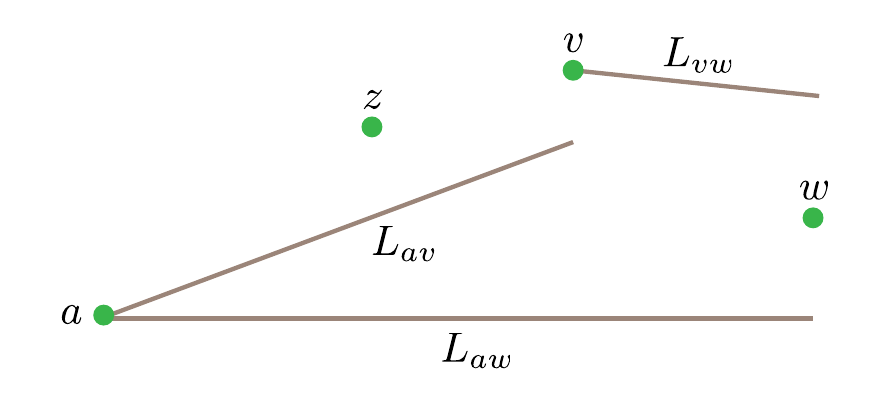}
\caption{}
\label{fig1}
\end{figure}

\subsection{Preliminary lemmas}

We will need the estimates from the following two lemmas in the proof of Lemma~\ref{TI}.

\begin{lemma}
\label{NH}
For any $a = (a_1,a_2)$ and $b = (b_1,b_2)$ in $\mathbb{G}$ with $|a| < 1$ and $|b| < 1$, we have
$$
\Vert NH(ab) \Vert^r \lesssim |a_1||b_1|\left| \frac{a_1}{|a_1|} - \frac{b_1}{|b_1|} \right| + |a_2|(|b_1|+|b_2|) + |b_2|(|a_1| + |a_2|) .
$$
\end{lemma}

In particular, we will use the fact that 
\begin{equation}
\label{NHNormBd}
\alpha \Vert NH(ab) \Vert^{2r} \leq \left(|a_1||b_1|\left| \frac{a_1}{|a_1|} - \frac{b_1}{|b_1|} \right|\right)^2 + \left(|a_2|(|b_1|+|b_2|) + |b_2|(|a_1| + |a_2|)\right)^2
\end{equation}
for some $0 < \alpha < 1$ depending only on the metric and group structure of $\mathbb{G}$.

\begin{proof}
We may write $a = e^A$ and $b=e^B$ for $A,B \in \mathfrak{g}$.
In other words, $A = \log a$ and $B = \log b$.
Write $A=A_1+A_2$ and $B=B_1+B_2$ where $A_1,B_1 \in V_1$ lie in the horizontal (first) layer of $\mathfrak{g}$.
According to the Baker-Campbell-Hausdorff formula (Theorem~\ref{BCH})
and the bilinearity of the Lie bracket, 
$\log (ab)$ is a finite sum of constant multiples of 
\begin{equation}
\label{Zbrackets}
[Z_1,[Z_2,\cdots,[Z_{k-2},[Z_{k-1},Z_k]]\cdots]]
\end{equation}
for some positive integer $k \leq r$ where each 
$Z_i$ is one of $A_1$, $B_1$, $A_2$, or $B_2$.
In particular, $[Z_{k-1},Z_k]$ must have the form
\begin{equation}
\label{brackets}
[A_1,B_1], \quad [A_1,B_2], \quad [A_2,B_1], \quad \text{or} \quad [A_2,B_2].
\end{equation}
Indeed,
we must have $s_n = 1$ or we have $s_n = 0$ and $r_n  = 1$
(for otherwise the brackets vanish),
so the nested brackets \eqref{Zbrackets} must have the form
$$
[\cdot,[\cdot,\cdots,[A,B]\cdots]] = [\cdot,[\cdot,\cdots,[A_1+A_2,B_1+B_2]\cdots]] = \sum_{i=1}^2 \sum_{j=1}^2 [\cdot,[\cdot,\cdots,[A_i,B_j]\cdots]]
$$
(since $[A,B]=-[B,A]$).

By definition, we have
$$
|NH(ab)| = |\tilde{\pi}(ab)^{-1} (ab)| = | (-a_1-b_1,0) (a_1 + b_1, a_2 + b_2 + P(a,b)) |
$$
for some polynomial $P$ (given by the BCH formula).
Thus by a similar argument as above, 
$\log (NH(ab))$ is a finite sum of constant multiples of nested brackets \eqref{Zbrackets}
each of which ends with a term of the form \eqref{brackets}.
Consider the norm $| \cdot|$ on $\mathfrak{g}$ 
induced by the Euclidean norm on the exponential coordinates $\mathbb{R}^N$.
(i.e. for $X \in \mathfrak{g}$ with $e^X = x \in \mathbb{G}$, we have $|X| = |x|$.)
Since 
$$
[A_1,B_1] = |A_1||B_1| \left[ \frac{A_1}{|A_1|},\frac{B_1}{|B_1|} \right]
= |A_1||B_1| \left[ \frac{A_1}{|A_1|} - \frac{B_1}{|B_1|},\frac{B_1}{|B_1|} \right],
$$
the bilinearity of the Lie bracket gives the following bound:
$$
|[A_1,B_1]| \lesssim |A_1||B_1|\left| \frac{A_1}{|A_1|} - \frac{B_1}{|B_1|} \right|.
$$
Thus, for those brackets \eqref{Zbrackets} ending with $[A_1,B_1]$, we have
$$
|[Z_1,[Z_2,\cdots,[Z_{k-2},[Z_{k-1},Z_k]]\cdots]]| 
\lesssim \left( \prod_{i=1}^{k-2} |Z_i| \right) \left|[A_1,B_1]\right|
\lesssim |a_1||b_1|\left| \frac{a_1}{|a_1|} - \frac{b_1}{|b_1|} \right|
$$
since $|a| < 1$ and $|b| < 1$.
All other nested brackets \eqref{Zbrackets} which do not end with $[A_1,B_1]$ satisfy
$$
|[Z_1,[Z_2,\cdots,[Z_{k-2},[Z_{k-1},Z_k]]\cdots]]|
\lesssim \prod_{i=1}^k |Z_i| 
%\leq |a_1||b_2| + |a_2||b_1| + |a_2||b_2| 
\leq |a_2|(|b_1|+|b_2|) + |b_2|(|a_1| + |a_2|).
$$
Since we may estimate $|NH(ab)| = |\log(NH(ab))|$ by a finite sum of constant multiples of the nested brackets \eqref{Zbrackets},
we have proven
$$
|NH(ab)| \lesssim |a_1||b_1|\left| \frac{a_1}{|a_1|} - \frac{b_1}{|b_1|} \right| + |a_2|(|b_1|+|b_2|) + |b_2|(|a_1| + |a_2|) .
$$
Hence there is some compact set $K \subset \mathbb{G}$
(depending only on the group structure and metric)
so that $NH(ab) \in K$ for any $a$ and $b$ in the Euclidean unit ball.
That is, we may apply \eqref{compact} to conclude
$\Vert NH(ab) \Vert^r \lesssim |NH(ab)|$.
This completes the proof.
\end{proof}

The following lemma is entirely Euclidean in nature and elementary.
The details of the proof are included for completeness.

\begin{lemma}
\label{height}
Fix $c,d \in \mathbb{R}^n$.
%Consider the triangle formed by $c$, $c+d$, and the origin.
Let $\ell_{c+d}$ denote the segment from the origin to $c+d$.
Then
$$
d_{\mathbb{R}^n}(c,\ell_{c+d})^2 \leq \tfrac12 |c||d| \left| \frac{c}{|c|} - \frac{d}{|d|} \right|^2.
$$
\end{lemma}

\begin{proof}
%The triangle has sides of length $|c|$, $|d|$, and $|c+d|$.
We will make frequent use of the following consequence of the polarization identity:
\begin{equation}
\label{polar}
|c||d| \left| \frac{c}{|c|} - \frac{d}{|d|} \right|^2 = |c|^2+2|c||d| + |d|^2 - |c+d|^2.
\end{equation}
Let $u$ denote the scalar projection of $c$ along $c+d$.
That is, $u = \frac{\langle c,c+d \rangle}{|c+d|}$.
If $u \leq 0$, then\footnote{Indeed, if $u \leq 0$, then the angle between the vector $c$ and $\ell_{c+d}$ is between $\pi / 2$ and $3 \pi / 2$.} 
the closest point in $\ell_{c+d}$ to $c$ is the origin, so $d_{\mathbb{R}^n}(c,\ell_{c+d}) = |c|$.
We then have\footnote{The assumption $u \leq 0$
implies $|c|^2 + \langle c,d \rangle
= \langle c,c \rangle + \langle c,d \rangle
= \langle c,c+d \rangle
\leq 0$.
Hence the polarization identity yields 
$2|c|^2 \leq - 2 \langle c,d \rangle = |c|^2  + |d|^2 - |c+d|^2$.}
$$
2d_{\mathbb{R}^n}(c,\ell_{c+d})^2 
= 2|c|^2
\leq |c|^2  + |d|^2 - |c+d|^2
\leq |c||d| \left| \frac{c}{|c|} - \frac{d}{|d|} \right|^2.
$$
If $u \geq |c+d|$, then the closest point in $\ell_{c+d}$ to $c$ is $c+d$.
Since\footnote{The polarization identity gives 
$|d|^2 = |c|^2 + |c+d|^2  - 2 \langle c,c+d \rangle
\leq |c|^2 + |c+d|^2 - 2|c+d|^2$ since the assumption $u \geq |c+d|$
implies $- 2 \langle c,c+d \rangle \leq -2|c+d|^2$.}
$|d|^2 \leq |c|^2 - |c+d|^2$, we have 
$$
2d_{\mathbb{R}^n}(c,\ell_{c+d})^2 
= 2|c - (c+d)|^2
=2|d|^2
=|d|^2 + |d|^2
\leq |c|^2 + |d|^2 - |c+d|^2
\leq |c||d| \left| \frac{c}{|c|} - \frac{d}{|d|} \right|^2.
$$

Now suppose $0 < u < |c+d|$.
That is, the projection of $c$ to the line containing $\ell_{c+d}$
actually lies in $\ell_{c+d}$.
Since this projection divides $\ell_{c+d}$ into segments of length $u$ and $|c+d|-u$,
the Pythagorean Theorem gives
\begin{align*}
|c||d| \left| \frac{c}{|c|} - \frac{d}{|d|} \right|^2 &= |c|^2+2|c||d| + |d|^2 - |c+d|^2 \\
& = \left[|c|^2 - u^2 \right] + \left[|d|^2 - (|c+d| - u)^2 \right] +2|c||d| +2u^2 - 2u|c+d| \\
& = d_{\mathbb{R}^n}(c,\ell_{c+d})^2 + d_{\mathbb{R}^n}(c,\ell_{c+d})^2 + 2|c||d| + 2u(u-|c+d|) \\
& = 2d_{\mathbb{R}^n}(c,\ell_{c+d})^2 + 2(|c||d| - u|d| + u|d| + u(u-|c+d|)) \\
& = 2d_{\mathbb{R}^n}(c,\ell_{c+d})^2 + 2|d|(|c| - u) + 2u(|d|- (|c+d|-u)) \\
& \geq 2d_{\mathbb{R}^n}(c,\ell_{c+d})^2
\end{align*}
since $(|c+d|-u)^2 \leq (|c+d|-u)^2 + d_{\mathbb{R}^n}(c,\ell_{c+d})^2 = |d|^2$
and $u \leq |c|$ by the Cauchy-Schwartz inequality.
\end{proof}

The technical proof of this next lemma follows the example of the proof from \cite{HebSik}
that the HS-norm is sub-linear.
By tightening some of the bounds from \cite[Theorem 2]{HebSik},
we are able to estimate the error in the sub-linearity of the norm.
Again, we have set $m=2^{-217}$, but this lemma actually holds for any $m \in [0,1]$.

\begin{lemma}
\label{TI}
Fix $x,y \in \mathbb{G}$.
Set
$$
A := \frac{m^{2r}}{16} \left( \frac{\alpha \left\Vert NH\left( xy \right) \right\Vert^{2r}}{(\Vert x \Vert + \Vert y \Vert)^{2r-1}} + \frac{d_{\mathbb{R}^n}(x_1,\ell_{x_1+y_1})^2 }{\Vert x \Vert + \Vert y \Vert} \right).
$$
If $4A \leq \min \{ \Vert x \Vert, \Vert y \Vert \}$,
then 
$
\Vert x \Vert + \Vert y \Vert - \Vert xy \Vert
\geq A.
$
\end{lemma}

\begin{proof}
Fix $x,y \in \mathbb{G}$.
We may assume that neither $x$ nor $y$ is 0
as, in that case, $A=0$ and the result trivially follows. 
Set $a = \delta_{1 / \Vert x \Vert} (x)$ and $b = \delta_{ 1 / \Vert y \Vert} (y)$
so that $|a|=|b|=\eta$.
Write 
$$
s = \frac{\Vert x \Vert}{\Vert x \Vert + \Vert y \Vert - A} 
\quad
\text{and}
\quad
t = \frac{\Vert y \Vert}{\Vert x \Vert + \Vert y \Vert - A}. 
$$
Note that $s,t \in [0,1]$ since $4A \leq \min \{ \Vert x \Vert, \Vert y \Vert \}$. 
These values have been chosen so that
$$
\left| \delta_{\frac{1}{\Vert x \Vert + \Vert y \Vert - A}}(xy) \right|
=
\left| \delta_{\frac{\Vert x \Vert}{\Vert x \Vert + \Vert y \Vert-A}} \left( \delta_{\frac{1}{\Vert x \Vert}} (x) \right) 
\cdot \delta_{\frac{\Vert y \Vert}{\Vert x \Vert + \Vert y \Vert-A}} \left( \delta_{\frac{1}{\Vert y \Vert}} (y) \right) \right|
=\left| \delta_{s} \left( a \right) \delta_{t} \left( b \right) \right|.
$$

%Following the proof from Hebisch and Sikkora \cite[Theorem 2]{HebSik},
Write $a=(a_1,a_2)$ and $b=(b_1,b_2)$ as before.
For ease of notation, we write write $\delta_{\lambda} (a_1) = \delta_{\lambda} ((a_1,0))$
and $\delta_{\lambda} (a_2) = \delta_{\lambda} ((0,a_2))$ for $\lambda > 0$, 
and we similarly use the shorthand $\delta_{\lambda} (b_1)$ and $\delta_{\lambda} (b_2)$.
Write $v = \delta_{s}(a_2) + \delta_{t}(b_2) + (0,R_2(\delta_{s} (a), \delta_{t}(b)))$ where $R_2$ is as defined in Section~\ref{CarnotSec}.
The bounds \eqref{HS2} and \eqref{HS3} give
\begin{align*}
|v| 
&\stackrel{\eqref{HS2}}{\leq} s^2|a_2|+t^2|b_2| + C_2\left( |\delta_s (a_1)||\delta_{t}(b_2)| + |\delta_s (a_2)||\delta_{t}(b_1)| + |\delta_s (a_2)||\delta_{t}(b_2)|\right) \\
&\;\leq s^2|a_2|+t^2|b_2| + st(C_2+1)\left( |a_2|(|b_1|+|b_2|) + |b_2|(|a_1| + |a_2|) \right) \\
&\hspace{1.5in} - \left( |\delta_s (a_2)|(|\delta_{t}(b_1)|+|\delta_{t}(b_2)|) + |\delta_{t}(b_2)|(|\delta_s (a_1)| + |\delta_s (a_2)|) \right) \\
& \stackrel{\eqref{HS3}}{\leq} s^2|a_2|+t^2|b_2| + \tfrac14 st \left(|a_2|+ |b_2|\right) \\
&\hspace{1.5in} - \left( |\delta_s (a_2)|(|\delta_{t}(b_1)|+|\delta_{t}(b_2)|) + |\delta_{t}(b_2)|(|\delta_s (a_1)| + |\delta_s (a_2)|) \right)\\
& \;\leq  s|a_2|+t|b_2| - \tfrac12 st \left(|a_2|+ |b_2|\right) \\
&\hspace{1.5in} - \left( |\delta_s (a_2)|(|\delta_{t}(b_1)|+|\delta_{t}(b_2)|) + |\delta_{t}(b_2)|(|\delta_s (a_1)| + |\delta_s (a_2)|) \right).
\end{align*}
This last inequality follows from the following argument:
since $4A \leq \Vert y \Vert$, we have 
$$
\tfrac34t - (1-s)
= \frac{\tfrac34 \Vert y \Vert}{\Vert x \Vert + \Vert y \Vert - A} 
- \frac{\Vert y \Vert - A}{\Vert x \Vert + \Vert y \Vert  - A}
= \frac{A-\tfrac14 \Vert y \Vert}{\Vert x \Vert + \Vert y \Vert  - A}
\leq 0
$$
so $(\tfrac34 t +s) \leq 1$ (and similarly $(\tfrac34 s +t) \leq 1$), and thus
$$
s^2|a_2|+t^2|b_2| -s|a_2| - t|b_2| + \tfrac34 st \left(|a_2|+ |b_2|\right) 
= s|a_2|\left((\tfrac34 t +s) - 1\right) + t|b_2|\left((\tfrac34s + t) - 1\right) \leq 0.
$$
Moreover, the inequalities $(\tfrac34 t +s) \leq 1$ and $(\tfrac34 s +t) \leq 1$ imply that
$$
|v| \leq s^2|a_2|+t^2|b_2| + \tfrac14 st \left(|a_2|+ |b_2|\right)
\leq |a_2|+ |b_2|.
$$
Hence
\begin{align*}
|v|^2  \left( 1+st \right) &\leq |v|^2 + st |v| \left(|a_2|+ |b_2|\right) \\
&\leq (|v| + \tfrac12 st \left(|a_2|+ |b_2|\right) )^2 \\
&\leq (s|a_2|+t|b_2|- \left( |\delta_s (a_2)|(|\delta_{t}(b_1)|+|\delta_{t}(b_2)|) + |\delta_{t}(b_2)|(|\delta_s (a_1)| + |\delta_s (a_2)|) \right)  )^2 \\
&\leq (s|a_2|+t|b_2|)^2 - \left( |\delta_s (a_2)|(|\delta_{t}(b_1)|+|\delta_{t}(b_2)|) + |\delta_{t}(b_2)|(|\delta_s (a_1)| + |\delta_s (a_2)|) \right)^2.
\end{align*}
This last inequality follows from the fact that $(k-\ell)^2 \leq k^2 - \ell^2$ when $\ell \leq k$ and
\begin{align*}
|\delta_s (a_2)|(|\delta_{t}(b_1)|+|\delta_{t}(b_2)|) + |\delta_{t}(b_2)|&(|\delta_s (a_1)| + |\delta_s (a_2)| \\
&\leq st\left( |a_2|(|b_1|+|b_2|) + |b_2|(|a_1| + |a_2|) \right) \\
&\leq s|a_2|+t|b_2|
\end{align*}
since $|a|=|b|=\eta < \frac12$.
Therefore %\eqref{HS1} and 
\eqref{polar} together with the fact that $2\langle u_1,u_2 \rangle \leq st|u_1|^2 + \frac{|u_2|^2}{st}$ for any $u_1,u_2 \in \mathbb{R}^N$ gives
\begin{align*}
|\delta_s (a) \delta_{t}(b)|^2
&= |\delta_s (a_1) + \delta_{t}(b_1)|^2 + |v+(0,R_1(\delta_s (a), \delta_{t}(b)))|^2 \\
&\leq (|\delta_s (a_1)| + |\delta_{t}(b_1)|)^2 
-|\delta_s (a_1)||\delta_{t} (b_1)|
\left| \frac{\delta_s (a_1)}{|\delta_s (a_1)|} - \frac{\delta_{t} (b_1)}{|\delta_{t} (b_1)|} \right|^2 \\
& \hspace{1in} + |v|^2\left( 1+st \right)
+ |R_1(\delta_s (a), \delta_{t}(b))|^2 \left(1+ \frac{1}{st} \right) \\
& \stackrel{\eqref{HS1}}{\leq} (s|a_1| + t|b_1|)^2 + (s|a_2| + t|b_2|)^2
- |\delta_s (a_1)||\delta_{t} (b_1)|
\left| \frac{\delta_s (a_1)}{|\delta_s (a_1)|} - \frac{\delta_{t} (b_1)}{|\delta_{t} (b_1)|} \right|^2 \\
& \hspace{1in} - \left( |\delta_s (a_2)|(|\delta_{t}(b_1)|+|\delta_{t}(b_2)|) + |\delta_{t}(b_2)|(|\delta_s (a_1)| + |\delta_s (a_2)|) \right)^2\\
& \hspace{1in} +  \left(1+ \frac{1}{st} \right) C_1^2 |\delta_s (a_1)|^2 |\delta_{t} (b_1)|^2 \left| \frac{\delta_s (a_1)}{|\delta_s (a_1)|} - \frac{\delta_{t} (b_1)}{|\delta_{t} (b_1)|} \right|^2
\end{align*}
Now, according to \eqref{NHNormBd}, we have
\begin{align*}
&-\left( |\delta_s (a_2)|(|\delta_{t}(b_1)|+|\delta_{t}(b_2)|) + |\delta_{t}(b_2)|(|\delta_s (a_1)| + |\delta_s (a_2)|) \right)^2 \\
&\hspace{1in} 
\leq \left( |\delta_s (a_1)| |\delta_{t} (b_1)| \left| \frac{\delta_s (a_1)}{|\delta_s (a_1)|} - \frac{\delta_{t} (b_1)}{|\delta_{t} (b_1)|} \right| \right)^2
-\alpha \, \Vert NH(\delta_{s}(a)\delta_{t}(b)) \Vert^{2r}
\end{align*}
and thus
\begin{align*}
|\delta_s (a) &\delta_{t}(b)|^2 \\
&\leq (s|a| + t|b|)^2 
- |\delta_s (a_1)||\delta_{t} (b_1)|
\left| \frac{\delta_s (a_1)}{|\delta_s (a_1)|} - \frac{\delta_{t} (b_1)}{|\delta_{t} (b_1)|} \right|^2 
\\
&\hspace{.5in} 
+  \left( \left(1+ \frac{1}{st} \right) C_1^2 + 1\right) |\delta_s (a_1)|^2 |\delta_{t} (b_1)|^2 \left| \frac{\delta_s (a_1)}{|\delta_s (a_1)|} - \frac{\delta_{t} (b_1)}{|\delta_{t} (b_1)|} \right|^2 
- \alpha \, \Vert NH(\delta_{s}(a)\delta_{t}(b)) \Vert^{2r} \\
&\leq (s|a| + t|b|)^2 
- \alpha \, \Vert NH(\delta_{s}(a)\delta_{t}(b)) \Vert^{2r}\\
&\hspace{.5in} 
+  \left(\left(  \left(st+ 1 \right) C_1^2 + st\right)|a_1||b_1| - 1 \right) |\delta_s (a_1)| |\delta_{t} (b_1)| \left| \frac{\delta_s (a_1)}{|\delta_s (a_1)|} - \frac{\delta_{t} (b_1)}{|\delta_{t} (b_1)|} \right|^2 \\
&\stackrel{\eqref{HS4}}{\leq} (s|a| + t|b|)^2 
- \alpha \, \Vert NH(\delta_{s}(a)\delta_{t}(b)) \Vert^{2r}
- \tfrac12 |\delta_s (a_1)| |\delta_{t} (b_1)| \left| \frac{\delta_s (a_1)}{|\delta_s (a_1)|} - \frac{\delta_{t} (b_1)}{|\delta_{t} (b_1)|} \right|^2\\
%&\;\leq (s|a| + t|b|)^2 
%- \alpha \, \Vert NH(\delta_{s}(a)\delta_{t}(b)) \Vert^{2r}
%- d_{\mathbb{R}^n}(\delta_s(a_1),\ell_{\delta_s(a_1)+\delta_t(b_1)})^2\\
&\;\leq (s|a| + t|b|)^2 
- \alpha \, \Vert NH(\delta_{s}(a)\delta_{t}(b)) \Vert^{2r}
- d_{\mathbb{R}^n}(sa_1,\ell_{sa_1+tb_1})^2 
\end{align*}
by Lemma~\ref{height}.
Since $sa_1 = \frac{x_1}{\Vert x \Vert + \Vert y \Vert -A}$, 
$tb_1 = \frac{y_1}{\Vert x \Vert + \Vert y \Vert -A}$, 
and $\left| \delta_{s} \left( a \right) \delta_{t} \left( b \right) \right| = \left| \delta_{\frac{1}{\Vert x \Vert + \Vert y \Vert - A}}(xy) \right|$,
we have
\begin{align*}
|\delta_s (a) \delta_{t}(b)| - (s|a| + t|b|) 
&\leq \frac{-\alpha \Vert NH(\delta_{s}(a)\delta_{t}(b)) \Vert^{2r} - d_{\mathbb{R}^n}(sa_1,\ell_{sa_1+tb_1})^2}
{|\delta_s (a) \delta_t(b)| + (s|a| + t|b|)}\\
&=\frac{-  \alpha \left\Vert NH\left( \delta_{\frac{1}{\Vert x \Vert + \Vert y \Vert  - A}}(xy) \right) \right\Vert^{2r}
- \frac{d_{\mathbb{R}^n}(x_1,\ell_{x_1+y_1})^2 }{(\Vert x \Vert + \Vert y \Vert  - A)^2} }
{\left| \delta_s (a) \delta_{t}(b) \right| + (s|a| + t|b|)} \\
&=\frac{-  \alpha \frac{\left\Vert NH\left( xy \right) \right\Vert^{2r}}{(\Vert x \Vert + \Vert y \Vert - A)^{2r}}  
- \frac{d_{\mathbb{R}^n}(x_1,\ell_{x_1+y_1})^2 }{(\Vert x \Vert + \Vert y \Vert  - A)^2} }
{\left| \delta_s (a) \delta_{t}(b) \right| + (s|a| + t|b|)} = (*).
\end{align*}
We proved in particular above
that $| \delta_s (a) \delta_{t}(b) | \leq s|a|+t|b|$.
Since $|a|=|b|=\eta$, this gives $| \delta_s (a) \delta_{t}(b) | \leq 2\eta < 8 m^{-2r} / \eta$.
Thus
\begin{align*}
(*)
&\leq -\frac{\eta m^{2r}}{16 } \left(   \frac{\alpha \left\Vert NH\left( xy \right) \right\Vert^{2r}}{(\Vert x \Vert + \Vert y \Vert -A)^{2r}}  
+ \frac{d_{\mathbb{R}^n}(x_1,\ell_{x_1+y_1})^2 }{(\Vert x \Vert + \Vert y \Vert  -A)^2} \right) \\
&\leq -\frac{\eta m^{2r}}{16} \left(   \frac{\alpha \left\Vert NH\left( xy \right) \right\Vert^{2r}}{(\Vert x \Vert + \Vert y \Vert )^{2r-1}}  
+ \frac{d_{\mathbb{R}^n}(x_1,\ell_{x_1+y_1})^2 }{\Vert x \Vert + \Vert y \Vert  } \right) (\Vert x \Vert + \Vert y \Vert  -A)^{-1}
\end{align*}

Since $|a|=|b|=\eta$, the definitions of $s$ and $t$ give
$$
s|a| + t|b| 
= \frac{\Vert x \Vert}{\Vert x \Vert + \Vert y \Vert  - A}|a| + \frac{\Vert y \Vert}{\Vert x \Vert + \Vert y \Vert  - A} |b|
= \frac{\eta(\Vert x \Vert + \Vert y \Vert)}{\Vert x \Vert + \Vert y \Vert - A}.
$$
In other words,
$$
\left| \delta_{\frac{1}{\Vert x \Vert + \Vert y \Vert - A}}(xy) \right|
=
|\delta_s (a) \delta_{t}(b)| 
\leq
\eta \left( \frac{ \Vert x \Vert + \Vert y \Vert 
- \frac{m^{2r}}{16} \left( \frac{\alpha \left\Vert NH\left( xy \right) \right\Vert^{2r}}{(\Vert x \Vert + \Vert y \Vert)^{2r-1}} + \frac{d_{\mathbb{R}^n}(x_1,\ell_{x_1+y_1})^2 }{\Vert x \Vert + \Vert y \Vert} \right)}
{\Vert x \Vert + \Vert y \Vert-A} \right)
= \eta
$$
according to the definition of $A$.
Therefore, the definition of the HS norm gives
$$
\Vert xy \Vert \leq \Vert x \Vert + \Vert y \Vert 
- \frac{m^{2r}}{16} \left( \frac{\alpha \left\Vert NH\left( xy \right) \right\Vert^{2r}}{(\Vert x \Vert + \Vert y \Vert)^{2r-1}} + \frac{d_{\mathbb{R}^n}(x_1,\ell_{x_1+y_1})^2 }{\Vert x \Vert + \Vert y \Vert} \right).
$$

\end{proof}

\begin{corollary}
\label{TICor}
Fix $v,w \in \mathbb{G}$.
If
\begin{equation}
\label{CorAss}
\frac{m^{2r}}{4} \left( \frac{\alpha \left\Vert NH\left( w \right) \right\Vert^{2r}}{(d(0,v) + d(v,w))^{2r-1}} + \frac{d_{\mathbb{R}^n}(v_1,\ell_{w_1})^2 }{d(0,v) + d(v,w)} \right)
\leq \min \{ d(0,v), d(v,w) \},
\end{equation}
then 
\begin{equation}
\label{CorCon}
\frac{m^{2r}}{16} \left( \frac{\alpha \left\Vert NH\left( w \right) \right\Vert^{2r}}{(d(0,v) + d(v,w))^{2r-1}} + \frac{d_{\mathbb{R}^n}(v_1,\ell_{w_1})^2 }{d(0,v) + d(v,w)} \right)
\leq d(0,v) + d(v,w) - d(0,w).
\end{equation}
\end{corollary}
\begin{proof}
Write $x=v$ and $y=v^{-1}w$ in Lemma~\ref{TI}.
\end{proof}

\subsection{Proof of Theorem~\ref{Goal}}

We are now ready to prove Theorem~\ref{Goal}.
This theorem controls the deviation of horizontal segments 
by the excess in a four point triangle inequality.
We restate it here for convenience.
Again, we have fixed $m = 2^{-217}$.%, but the results in this section hold for any $0 < m < 1$.
\begin{reptheorem}{Goal}
%Fix $0<m<1$. (For our purposes later, we will use $m = 2^{-217}$.)
Suppose $a,z,v,w \in \mathbb{G}$ satisfy
$$
m \rho \leq \min \{ d(a,z), d(a,v), d(z,v), d(v,w) \}
$$
and
$$
\max \{ d(a,z), d(a,v), d(z,v), d(v,w), d(a,w) \} \leq \rho
$$
for some $\rho > 0$.
Then there is a constant $C_0 = C_0(\mathbb{G}) > 0$ so that
$$
\sup_{t \in [0,1]} d(L_{av}(t),L_{aw})^{2r^2} + \sup_{t \in [0,1]} d(L_{vw}(t),L_{aw})^{2r^2}
\leq C_0 \rho^{2r^2-1} \Delta
$$
where $\Delta := d(a,z) + d(z,v) + d(v,w) - d(a,w)$.
\end{reptheorem}

We may assume without loss of generality from here on out that $a=0$.
Indeed, the metric is left invariant, and horizontal segments commute with left multiplication in the following sense:
$$
c^{-1} L_{ab}(t) 
= c^{-1} a \delta_t (\tilde{\pi}(a^{-1}b))
= c^{-1} a \delta_t (\tilde{\pi}((c^{-1}a)^{-1}(c^{-1}b)))
=L_{(c^{-1}a)(c^{-1}b)}(t)
$$
for any $b,c \in \mathbb{G}$ and any $t \in [0,1]$.
%In this theorem, $\rho$ is the scale at which we are working.
We will first establish the important tools used in the proof of the theorem.

\begin{lemma}
Under the assumptions of Theorem~\ref{Goal}, we have
\begin{equation}
\label{Delta}
\rho^{-(2r-1)} \Vert NH(w) \Vert^{2r} + \rho^{-1} d_{\mathbb{R}^n}(v_1,\ell_{w_1})^2 \lesssim d(0,v) + d(v,w) - d(0,w)
\end{equation}
and
\begin{equation}
\label{Delta2}
\rho^{-(2r-1)} \Vert NH(v) \Vert^{2r} + \rho^{-1} d_{\mathbb{R}^n}(z_1,\ell_{v_1})^2 \lesssim d(0,z) + d(z,v) - d(0,v).
\end{equation}
\end{lemma}
\begin{figure}[H]
\centering
\includegraphics[scale = 1]{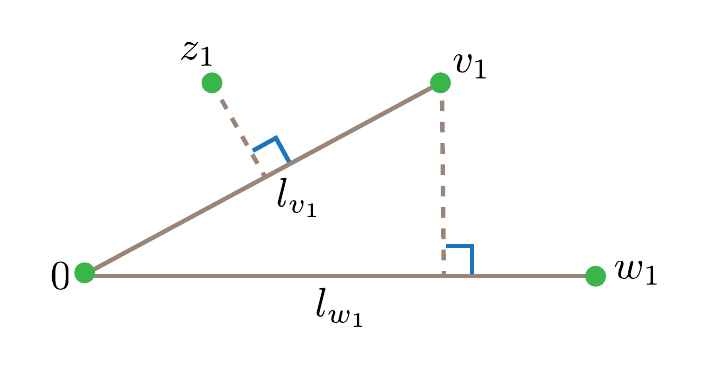}
\caption{}
\label{fig1}
\end{figure}
\begin{proof}
First, 
$\Vert NH(w) \Vert \leq \Vert \tilde{\pi}(w) \Vert + \Vert w \Vert \leq 2d(0,w) \leq 2\rho$, so 
$$
\frac{\alpha \left\Vert NH\left( w \right) \right\Vert^{2r}}{(d(0,v) + d(v,w))^{2r-1}}
\leq \frac{ (2\rho)^{2r}}{(2m\rho)^{2r-1}} 
\leq  2 m^{-2r} \min \{ d(0,v), d(v,w) \}.
$$
We also
have
$d_{\mathbb{R}^n}(v_1,\ell_{w_1}) \leq |v_1 - w_1| = \eta \Vert \tilde{\pi}(v^{-1}w) \Vert \leq \Vert v^{-1}w \Vert = d(v,w)$, so
\begin{align*}
\frac{d_{\mathbb{R}^n}(v_1,\ell_{w_1})^2}{d(0,v) + d(v,w)}
\leq \frac{ \rho^2}{2m\rho} 
&\leq  \frac{1}{2m^2} \min \{ d(0,v), d(v,w) \} 
< 2 m^{-2r} \min \{ d(0,v), d(v,w) \} .
\end{align*}
We have just shown that $v$ and $w$ satisfy the assumption \eqref{CorAss}.
Thus Corollary~\ref{TICor} gives \eqref{Delta}.
Identical arguments show that $z$ and $v$ satisfy \eqref{CorAss}
with $z$ in the place of $v$ and $v$ in the place of $w$.
This gives \eqref{Delta2} and completes the proof of the lemma.
\end{proof}

Consider the following from \cite[Lemma~3.8]{LiThesis}:
\begin{lemma}
\label{SegLem}
Fix a constant $M \geq 1$.
For any $0 < \omega < M$,
if $f,g:[0,1] \to \mathbb{G}$
are $\rho$-Lipschitz, constant speed, horizontal segments satisfying
$$
d(f(0),g(0)) \leq \omega \rho \quad \text{and} \quad d(f(1),g(1)) \leq \omega \rho,
$$
then $\sup_{t \in [0,1]} d(f(t),g(t)) \leq C' \omega^{1/r} \rho$
for some constant $C'=C'(\mathbb{G},M)>0$.
\end{lemma}

In other words, 
as long as the endpoints of the horizontal segment $f$ are close enough to the endpoints of $g$,
then any point along $f$ is close to $g$ (up to a factor of a different power).
Note that the original lemma in \cite{LiThesis} is stated for $0 < \omega < 1$,
but the triangle inequality gives
\begin{align*}
d(f(t),g(t)) 
\leq d(f(t),f(0)) + d(f(0),g(0)) + d(g(0),g(t))
&\leq \rho + \omega \rho + \rho \\
&\leq 3\omega \rho
\leq (3 M^{1-\frac{1}{r}}) \omega^{1/r} \rho.
\end{align*}
when $1 \leq \omega < M$.
Lemma~\ref{SegLem} translates in our case to the following:
\begin{lemma}
\label{BoundLem}
Suppose the following bounds hold for some $C=C(\mathbb{G})>0$:
\begin{align}
d(L_v(1),L_w)^{2r} = d(\tilde{\pi}(v),L_w)^{2r} &\leq C\rho^{2r-1} \Delta, \label{i} \\
d(L_{vw}(0),L_w)^{2r} = d(v,L_w)^{2r} &\leq C\rho^{2r-1} \Delta, \label{ii} \\
d(L_{vw}(1),L_w)^{2r} = d(v \, \tilde{\pi}( v^{-1} w) ,L_w)^{2r} &\leq C\rho^{2r-1} \Delta. \label{iii}
\end{align}
Then 
$$
\sup_{t \in [0,1]} d(L_v(t),L_w)^{2r^2} + \sup_{t \in [0,1]} d(L_{vw}(t),L_w)^{2r^2}
\lesssim \rho^{2r^2-1} \Delta.
$$
\end{lemma}

\begin{proof}
Say $f = L_{vw}$ and $\hat{g} = L_w$.
(The case of $f = L_v$ is similar and simpler.)

According to \eqref{ii} and \eqref{iii}, we have
$$
d(f(0),\hat{g}(t_1)) \leq C^{\frac{1}{2r}} \rho^{\frac{2r-1}{2r}} \Delta^{\frac{1}{2r}} = \rho(C\rho^{-1}\Delta)^{\frac{1}{2r}}
\quad \text{and} \quad
d(f(1),\hat{g}(t_2)) 
\leq %C \rho^{\frac{2r-1}{2r}} \Delta^{\frac{1}{2r}} =
\rho(C\rho^{-1}\Delta)^{\frac{1}{2r}}
$$
for some $t_1,t_2 \in [0,1]$.
Note that $f$ and $\hat{g}$ are both $\rho$-Lipschitz.
Indeed, for any $s,t \in [0,1]$,
\begin{align*}
d(f(s),f(t)) = d(v*\delta_s(\tilde{\pi}(v^{-1}w)),v*\delta_t(\tilde{\pi}(v^{-1}w))) 
&= d((s(w_1-v_1),0),(t(w_1-v_1),0)) \\
&= \Vert (-t(w_1-v_1),0) * (s(w_1-v_1),0) \Vert \\
&= \Vert ((s-t)(w_1-v_1),0) \Vert \\
&= |s-t| \Vert \tilde{\pi}(v^{-1}w) \Vert  \\
&\leq |s-t| d(v,w) \leq |s-t| \rho.
\end{align*}
The fourth equality follows from the fact that 
the brackets in the BCH formula all vanish in this particular product since 
$\delta_s(\tilde{\pi}(v^{-1}w))$ and $\delta_t(\tilde{\pi}(v^{-1}w))$
are co-linear in $\mathbb{R}^n \times \{0\}$
and $[X,X]=0$.
%and the first inequality follows from the definition of the HS norm as in the above footnote.
Similarly, one may check that $\hat{g}$ is $\rho$-Lipschitz.
%\begin{align*}
%d(g(s),g(t)) = d(\delta_s(\tilde{\pi}(w)),\delta_t(\tilde{\pi}(w))) &= d((sw_1,0),(tw_1,0)) \\
%&= \Vert (-tw_1,0) * (sw_1,0) \Vert \\
%&= \Vert ((s-t)w_1,0) \Vert \\
%&= |s-t| \Vert \tilde{\pi}(w) \Vert  \leq |s-t| d(0,w) \leq |s-t| \rho.
%\end{align*}

Now define $g:[0,1] \to \mathbb{G}$ so that $g(t) = \hat{g}(t(t_2-t_1) + t_1)$.
(We assume without loss of generality that $t_1 \leq t_2$.
If not, swap the roles of $t_1$ and $t_2$ in the definition of $g$.)
Then $g$ is $\rho$-Lipschitz, $g(0) = \hat{g}(t_1)$, $g(1) = \hat{g}(t_2)$, and
$$
g(t) 
%= \hat{g}(t(t_2-t_1) + t_1) 
= \delta_{t(t_2-t_1) + t_1}(\tilde{\pi}(w))
=\delta_{t_1}(\tilde{\pi}(w)) * \delta_t( \delta_{t_1}(\tilde{\pi}(w))^{-1} \delta_{t_2}(\tilde{\pi}(w)) ).
$$
(The last equality holds since the points are co-linear in $\mathbb{R}^n \times \{ 0 \}$ (as above).)
In other words, $g$ is the horizontal segment from $\delta_{t_1}(\tilde{\pi}(w))$ to $\delta_{t_2}(\tilde{\pi}(w))$.
The hypotheses of Lemma~\ref{SegLem} are satisfied with $\omega = (C\rho^{-1}\Delta)^{\frac{1}{2r}} \leq (3C)^{\frac{1}{2r}}$,
so 
$$
d(L_{vw}(t),L_w) \leq d(f(t),g(t)) \leq C' (\rho^{-1}\Delta)^{\frac{1}{2r^2}} \rho
$$
for some $C'=C'(\mathbb{G})>0$ and any $t \in [0,1]$.
Therefore
$$
d(L_{vw}(t),L_w)^{2r^2} \lesssim \rho^{2r^2 - 1} \Delta.
$$
\end{proof}

We now prove Theorem~\ref{Goal}.
Our main tool will be the application of inequalities \eqref{Delta} and \eqref{Delta2}
to prove \eqref{i}, \eqref{ii}, and \eqref{iii}.
We will show that the distance of each endpoint to the segment $L_w$
is controlled by the
distance of $\tilde{\pi}(v)$ to the segment in the horizontal layer
plus the nonhorizontal part of $v$.
\begin{proof}[Proof of Theorem~\ref{Goal}]
It suffices to prove \eqref{i}, \eqref{ii}, and \eqref{iii}.
We start with \eqref{i}. 
Note that $\tilde{\pi}(v) \in \mathbb{R}^n \times \{0\}$ and $L_w \subset \mathbb{R}^n \times \{0\}$. 
Say $p_v \in L_w$ is the nearest point in $L_w$ to $\tilde{\pi}(v)$ (in the Euclidean norm).
Therefore, since $\delta_s(p) = sp$ for any $p \in \mathbb{R}^n \times \{ 0 \}$, we have
\begin{align*}
\frac{d(\tilde{\pi}(v),p_v)^{2r} }{\rho^{2r-1}} 
= \rho \left( \frac{d(\tilde{\pi}(v),p_v)}{\rho} \right)^{2r}
= \rho \, d(\delta_{1/\rho} (\tilde{\pi}(v)),\delta_{1/\rho} (p_v))^{2r}
&\lesssim \rho \, |\delta_{1/\rho} (\tilde{\pi}(v)) - \delta_{1/\rho} (p_v)|^{2} \\
&= \rho^{-1} d_{\mathbb{R}^n}(v_1, \ell_{w_1})^{2}
\lesssim \Delta
\end{align*}
according to \eqref{compact} and \eqref{Delta}.
(Note that the constant from \eqref{compact} here depends only on $\mathbb{G}$ 
since $\delta_{1/\rho} (\tilde{\pi}(v))$ and $\delta_{1/\rho} (p_v)$ lie in the unit ball of $\mathbb{G}$).
This gives \eqref{i}. 

Let us now prove \eqref{ii}.
Arguing as above using \eqref{Delta} and \eqref{Delta2} gives
\begin{align*}
\frac{d(v,L_w)^{2r} }{\rho^{2r-1}} 
&\lesssim
\frac{d(v,\tilde{\pi}(v))^{2r} }{\rho^{2r-1}} 
+
\frac{d(\tilde{\pi}(v),p_v)^{2r} }{\rho^{2r-1}} \\
&=
\frac{\Vert NH(v) \Vert^{2r} }{\rho^{2r-1}} 
+
\frac{d(\tilde{\pi}(v),p_v)^{2r} }{\rho^{2r-1}} \\
&\lesssim
\left[ \frac{\Vert NH(v) \Vert^{2r} }{\rho^{2r-1}} 
+
\frac{d_{\mathbb{R}^n}(z_1, \ell_{v_1})^{2}}{\rho} \right]
+
\left[ \frac{\Vert NH(w) \Vert^{2r} }{\rho^{2r-1}} 
+
\frac{d_{\mathbb{R}^n}(v_1, \ell_{w_1})^{2} }{\rho} \right] \\
&\lesssim 
[d(0,z) + d(z,v) - d(0,v)] + [d(0,v) + d(v,w) - d(0,w)] 
= \Delta.
\end{align*}

We will now prove \eqref{iii}.
Writing $v = \tilde{\pi}(v) NH(v)$ gives
\begin{align*}
\rho^{-1} d(v\tilde{\pi}(v^{-1}w),\tilde{\pi}(w)) 
&= \rho^{-1} d(\tilde{\pi}(v) NH(v) \tilde{\pi}(v^{-1}w),\tilde{\pi}(w)) \\
&= \Vert \delta_{1/ \rho}(\tilde{\pi}(v^{-1}w) \tilde{\pi}(w)^{-1} \tilde{\pi}(v)) \delta_{1 / \rho} (NH(v)) \Vert.
\end{align*}
Since the HS norm is invariant under rotations of $\mathbb{R}^n \times \{ 0 \} \subset \mathbb{R}^N$
which fix the other $N-n$ coordinates,
we may assume without loss of generality that the segment $L_w$ lies along the $x_1$ axis in $\mathbb{R}^N$.
Under this assumption, we have 
$$
\tilde{\pi}(w) = (w_1,0) = (w_1^1,0,0) 
\quad \text{and} \quad
\tilde{\pi}(v) = (v_1,0) = (v_1^1,v_1^2,0)
$$
where $w_1^1,v_1^1 \in \mathbb{R}$ and $v_1^2 \in \mathbb{R}^{n-1}$.
In particular, it follows that $d_{\mathbb{R}^n}(v_1,\ell_{w_1}) = |v_1^2|$.
We can thus write
$$
\delta_{1/ \rho}(\tilde{\pi}(v^{-1}w) \tilde{\pi}(w)^{-1} \tilde{\pi}(v))
= \delta_{1/ \rho}( (w_1^1 - v_1^1, -v_1^2, 0)(-w_1^1,0,0)(v_1^1, v_1^2,0))
= (0, p_2 )
$$
where $p_2$ is some BCH polynomial.
The bound \eqref{compact} gives
$$
\frac{d(v\tilde{\pi}(v^{-1}w),\tilde{\pi}(w))^{2r}}{\rho^{2r-1}}
\lesssim \rho \Vert \delta_{1/\rho}(NH(v)) \Vert^{2r} + \rho \Vert (0,p_2) \Vert^{2r}
\lesssim \rho \Vert \delta_{1/\rho}(NH(v)) \Vert^{2r} + \rho | (0,p_2) |^{2}.
$$
(Again, the constant from \eqref{compact} here depends only on $\mathbb{G}$ 
since $(0,p_2) \in B(0,3)$).
Arguing as in the proof of Lemma~\ref{NH}, 
we may see that the polynomial $p_2$ is a finite sum of constant multiples of 
terms of the form 
\begin{equation}
\label{brackets2}
[Z_1,[Z_2,\cdots,[Z_{k-2},[Z_{k-1},Z_k]]\cdots]]
\end{equation}
where each $Z_i$ is either $\rho^{-1} w_1^1$, $\rho^{-1}v_1^1$, or $\rho^{-1}v_1^2$.
(Note the abuse of notation here in which we identify $(w_1^1,0,0)$, $(v_1^1,0,0)$ and $(0,v_1^2,0)$ 
with their associated vectors in the first layer of the Lie algebra.)
Note that the point $(w_1^1,0,0)$ is simply a dilation of the point $(v_1^1,0,0)$.
Hence the associated vector $w_1^1$ is a constant multiple of the vector $v_1^1$.
That is, $[w_1^1,v_1^1] = 0$.
In particular, it follows that each term of the form \eqref{brackets2} can only end with
$$
[\rho^{-1} w_1^1,\rho^{-1} v_1^2] \quad
\text{ or } \quad [\rho^{-1} v_1^1,\rho^{-1} v_1^2].
$$
Since $\eta < 1$, we have
$\rho^{-1}|w_1^1| = \rho^{-1} \eta \Vert \tilde{\pi}(w) \Vert \leq \rho^{-1}d(0, w) \leq 1$,
and similarly we get
$\rho^{-1}|v_1^1| \leq 1$
and $\rho^{-1}|v_1^2| \leq 1$.
Therefore,
$$
|[Z_1,[Z_2,\cdots,[Z_{k-2},[Z_{k-1},Z_k]]\cdots]]| \leq \prod_{i=1}^{k} |Z_i| \leq \rho^{-1}|v_1^2| = \rho^{-1}d_{\mathbb{R}^n}(v_1,\ell_{w_1})
$$
for each term of the form \eqref{brackets2}.
Hence we may conclude
$$
\frac{d(v\tilde{\pi}(v^{-1}w),L_w)^{2r}}{\rho^{2r-1}}
\lesssim \rho \Vert \delta_{1/\rho}(NH(v)) \Vert^{2r} + \rho ( \rho^{-1}d_{\mathbb{R}^n}(v_1,\ell_{w_1}) )^{2}
= \frac{\Vert NH(v) \Vert^{2r}}{\rho^{2r-1}} + \frac{d_{\mathbb{R}^n}(v_1,\ell_{w_1})^{2}}{\rho}
$$
which, as in the proof of \eqref{ii}, is bounded by a constant multiple of $\Delta$.
This concludes the proof of \eqref{iii}. 
We have therefore proven the hypotheses of Lemma~\ref{BoundLem} and thus the theorem.
\end{proof}

\section{Using Theorem~\ref{Goal} to prove Theorem~\ref{TSP}}
\label{TSTSec}

We will now apply the estimates from Section~\ref{curvatureSec} 
to prove the Traveling Salesman Theorem (Theorem~\ref{TSP}) in $\mathbb{G}$.
In this section,
we will follow the proof of the Traveling Salesman Theorem in the Heisenberg group \cite[Theorem I]{LiSchul}
given in \cite{LiSchul}.
Many of the arguments therein hold in any metric space.
As such, this section will provide a rough outline of the proof of Theorem~\ref{TSP}.
Full proofs will be provided for the results whose proofs differ significantly from those in \cite{LiSchul}.

\subsection{Preliminaries: arcs}

First, we will recall the notation from \cite{LiSchul}.
Fix a connected $\Gamma \subset \mathbb{G}$ with $\mathcal{H}^1(\Gamma) < \infty$
and a 1-Lipschitz, arc-length parameterization $\gamma:\mathbb{T} \to \Gamma$
(where $\mathbb{T}$ is a circle in $\mathbb{R}^2$).
Such a parameterization exists according to Lemma~2.10 in \cite{LiSchul}.
Orient $\mathbb{T}$ so that we may discuss a particular direction of flow along $\Gamma$.
Since $\beta_{\Gamma}$ is scale invariant (i.e. $\beta_{\delta_{\lambda}(\Gamma)}(\delta_{\lambda}(B)) = \beta_{\Gamma}(B)$),
we may assume without loss of generality that $\diam(\Gamma) = 1$.

An \emph{arc} $\tau$ in $\gamma$ is 
the restriction $\gamma|_{I_{\tau}}$
where $I_{\tau} = [a(\tau),b(\tau)]$ is some interval in $\mathbb{T}$ compatible with the orientation chosen above.
Given two arcs $\tau$ and $\zeta$, the notation $\zeta \subset \tau$ means $I_{\zeta} \subset I_{\tau}$,
and we will write $\diam(\tau)$ to represent $\diam(\tau(I_{\tau}))$.

For any $L>0$, we define a \emph{prefiltration} $\mathcal{F}^0 = \bigcup_{n} \mathcal{F}_n^0$ to be 
a collection of arcs in $\gamma$ satisfying the following three conditions for any $n \in \mathbb{N}$:
\begin{enumerate}
\item For $\tau \in \mathcal{F}_n^0$, we have $L2^{-100n} \leq \diam(\tau) <L2^{-100n+3}$.
\item The domains of any two distinct arcs in $\mathcal{F}_n^0$ are disjoint in $\mathbb{T}$.
\item For any $k \in \mathbb{N}$, if the domains of the arcs $\zeta \in \mathcal{F}_{n+k}^0$ and $\tau \in \mathcal{F}_n^0$ intersect non-trivially, then $\zeta \subset \tau$.
\end{enumerate}
%In other words, a prefiltration of $\gamma$ is a nested collection of intervals in the circle so that the diameter of the image of any interval under $\gamma$ is controlled by the level of the collection.
According to \cite[Lemma 2.13]{LiSchul},
given any prefiltration $\mathcal{F}^0$,
one may construct a \emph{filtration} $\mathcal{F} = \cup_{n} \mathcal{F}_n$
generated by $\mathcal{F}^0$
i.e. a collection of arcs in $\gamma$ satisfying the following for any $n \in \mathbb{N}$:
\begin{enumerate}
\item Given $\zeta \in \mathcal{F}_{n+1}$, there is a unique $\tau \in \mathcal{F}_n$ such that $\zeta \subset \tau$.
\item For $\tau \in \mathcal{F}_n$, we have $L 2^{-100n-10} \leq \diam(\tau)<L2^{-100n+4}$.
\item The domains of any two distinct arcs in $\mathcal{F}_n$ are either disjoint or intersect in (one or both of) their endpoints.
\item $\bigcup_{\tau \in \mathcal{F}_n} \tau = \mathbb{T}$.
\item For each arc $\tau^0 \in \mathcal{F}_n^0$, there is a unique arc $\tau \in \mathcal{F}_n$ such that $\tau^0 \subset \tau$. Moreover, if $I_0$ and $I$ are the domains of $\tau_0$ and $\tau$ respectively, then the image of each of the connected components of $I \setminus I_0$ under $\gamma$ has diameter less than $L 2^{-100n-10}$.
\end{enumerate}

\subsection{Preliminaries: balls}

For each $n \in \mathbb{Z}$, choose a $2^{-n}$ separated net $\mathbb{X}_n$ of $\Gamma$
(i.e. a set $\mathbb{X}_n \subset \Gamma$ such that $d(x,y) \geq 2^{-n}$ for any $x,y \in \mathbb{X}_n$, 
and such that, for any $z \in \Gamma$, there is some $x \in \mathbb{X}$ with $d(z,x) < 2^{-n}$).
Define a \emph{multiresolution} of $\Gamma$ as follows:
$$
\hat{\mathcal{G}} = \{ B(x,10/2^n) \, : \, x \in \mathbb{X}_n \text{ and } n \in \mathbb{Z} \}.
$$
We will use \cite[Lemma 2.6]{LiSchul} 
(which holds here with the same proof since $\mathbb{G}$ is $Q$-regular)
to prove the Traveling Salesman Theorem (Theorem~\ref{TSP}) 
by establishing the bound
\begin{equation}
\label{TSPdiscrete}
\sum_{B \in \hat{\mathcal{G}}} \beta_{\Gamma}(B)^{2r^2} \diam(B) \leq C \mathcal{H}^1(\Gamma)
\end{equation}
where $C$ depends only on $\mathbb{G}$.
As in \cite[Lemma 2.9]{LiSchul},
it suffices to prove inequality \eqref{TSPdiscrete} 
when the sum is taken over the family $\mathcal{G}$ of balls in $\hat{\mathcal{G}}$ with radius less than 1/100. 

For a ball $B=B(x,r)$, write $\alpha B = B(x,\alpha r)$.
Fix an integer $\kappa > 3$.
Define $\mathcal{B}$ to be
the collection of balls $2B$ where $B \in \mathcal{G}$.
According to \cite[Lemma 2.11]{LiSchul},
since $\mathbb{G}$ is doubling,
there is a constant $D=D(\kappa)>0$ 
and a decomposition $\mathcal{B} = \cup_{i=1}^{D} \mathcal{B}_i$
into pairwise disjoint families of balls satisfying the following for each $i$:
\begin{enumerate}
\item if $B_1,B_2 \in \mathcal{B}_i$ have the same radius $r$, then $d(B_1,B_2) > \kappa r$. %(Check kappa. Was it 3 or 4???)
\item for any $B_1,B_2 \in \mathcal{B}_i$, the ratio of their radii equals $2^{100j}$ for some $j \in \mathbb{Z}$.
\end{enumerate}

Fix $i \in \{1,\dots,D\}$.
From each ball $B \in \mathcal{G}$ with $2B \in \mathcal{B}_i$, 
we may construct a set $Q(B)$ (called a \emph{cube})
so that the family $\Delta(\mathcal{B},i)$ of cubes constructed from double-balls in $\mathcal{B}_i$
satisfies the following (see \cite[Lemma 2.12]{LiSchul}):
\begin{enumerate}
\item $2B \subset Q(B) \subset (1+2^{-98})2B$.
\item Fix $2B,2B' \in \mathcal{B}_i$. If $Q(B) \cap Q(B') \neq \emptyset$ and the radius of $B$ is larger than the radius of $B'$, then $Q(B') \subset Q(B)$.
\item Fix balls $2B,2B' \in \mathcal{B}_i$ of equal radius. Then $d(Q(B),Q(B')) > (\kappa - 1) r$. %???Kappa issue again (2 is kappa - 1).
\end{enumerate}
Given any cube $Q(B) \in \Delta(\mathcal{B},i)$, define
$$
\Lambda(Q(B)) = \{ \tau = \gamma|_I \, : \, I \text{ is a connected component of } \gamma^{-1}(\Gamma \cap Q(B)) \text{ and } \gamma(I) \cap B \neq \emptyset \}. 
$$ 
These are the arcs of $\gamma$ inside $Q(B)$ that meet $B$.
According to \cite[Lemma 2.17]{LiSchul},
the collection of arcs $\mathcal{F}^{0,i} = \bigcup_{Q(B) \in \Delta(\mathcal{B},i)} \Lambda(Q(B))$
is a prefiltration for some $L_i >0$.
As discussed above, this induces a filtration $\mathcal{F}^i$.
In particular, for each $\tau^0 \in \mathcal{F}^{0,i}$,
there is a unique $\tau \in \mathcal{F}^{i}$ with $\tau^0 \subset \tau$.
We can therefore define the collection of extensions of arcs in $\Lambda(Q(B))$ as
$$
\Lambda'(Q(B)) = \{ \tau \in \mathcal{F}^{i} \, : \, \tau \supset \tau^0 \text{ and } \tau^0 \in \Lambda(Q(B)) \}
$$
for each cube $Q(B) \in \Delta(\mathcal{B},i)$.

For any arc $\tau$ of $\gamma$ with domain $I_{\tau}$, write $L_{\tau} = L_{a(\tau)b(\tau)}$ and define 
$$
\beta(\tau) 
= \sup_{t \in I_{\tau}} \frac{d(\gamma(t),L_{a(\tau)b(\tau)})}{\diam(\tau)}
= \sup_{p \in \tau} \frac{d(p,L_\tau)}{\diam(\tau)}.
$$
As in \cite{LiSchul}, write $\mathcal{G} = \mathcal{G}_1 \cup \mathcal{G}_2$ where
$$
\mathcal{G}_1 = \{ B \in \mathcal{G} \, : \, \text{there is } \tau \in \Lambda'(Q(B)) \text{ such that } \beta(\tau) \geq 10^{-10} \beta_{\Gamma}(B) \}
$$
and
$$
\mathcal{G}_2 = \{ B \in \mathcal{G} \, : \, \beta(\tau) < 10^{-10} \beta_{\Gamma}(B) \text{ for all } \tau \in \Lambda'(Q(B)) \}.
$$

\subsection{Non-flat balls}
In this section, we prove 
\begin{proposition}
\label{TSPdiscreteG1}
$$
\sum_{B \in \mathcal{G}_1} \beta_{\Gamma}(B)^{2r^2} \diam(B) \lesssim \mathcal{H}^1(\Gamma).
$$
\end{proposition}
This is half of the estimate \eqref{TSPdiscrete} (and thus half of the proof of Theorem~\ref{TSP}).
We first need to introduce some notation.
There are $D$ different filtrations of $\gamma$ to consider, but we will treat them individually.
Fix $i \in \{1,\dots,D\}$ and write $\mathcal{F} := \mathcal{F}^i$.
Recall that $\mathcal{F} = \bigcup_k \mathcal{F}_k$
by the definition of a filtration.
Given $\tau \in \mathcal{F}_k$ and $j \in \mathbb{N}$,
write 
$$
\mathcal{F}_{\tau,j} = \{ \tau' \in \mathcal{F}_{k+j} \, : \, \tau' \subset \tau \}.
$$
This is the collection of arcs $j$ layers lower in the filtration which are contained in $\tau$.
Define
$$
d_{\tau} = \max_{\tau' \in \mathcal{F}_{\tau,1}} \sup_{z \in L_{\tau'}} d(z,L_{\tau})
\quad
\text{for any }
\tau \in \mathcal{F}.
$$
According to \cite[Lemma 3.4]{LiSchul},
$d_{\tau} \leq 2 \diam(\tau)$.
We now prove the following version of Lemma~3.5 in \cite{LiSchul}.
This is the first place in this section where our proof differs significantly from the arguments in \cite{LiSchul},
so details are included.
In particular, it is in the proof of this lemma that we use the curvature bound from Theorem~\ref{Goal}.

\begin{lemma}
For any $\tau \in \mathcal{F}$, we have
\begin{equation}
\label{NewLem}
\frac{d_{\tau}^{2r^2}}{\diam(\tau)^{2r^2-1}}
\leq 
C'' \left( \left( \sum_{\tau' \in \mathcal{F}_{\tau,3}} d(\gamma(a(\tau')),\gamma(b(\tau'))) \right) - d(\gamma(a(\tau)),\gamma(b(\tau))) \right)
\end{equation}
for some $C'' = C''(\mathbb{G}) >0$.
\end{lemma}

\begin{proof}
Fix some $\tau \in \mathcal{F}_k \subset \mathcal{F}$.
As in the proof of Lemma~3.5 in \cite{LiSchul}, we have
\begin{equation}
\label{firstBound}
\frac{d_{\tau}^{2r^2}}{\diam (\tau)^{2r^2-1}}
\leq 2^{2r^2} \diam(\tau)
\leq 2^{2r^2} L 2^{-100k+4}
= 2^{2r^2+117} L 2^{-100k-113}.
\end{equation}
Thus if 
$$
\left( \left( \sum_{\tau' \in \mathcal{F}_{\tau,3}} d(\gamma(a(\tau')),\gamma(b(\tau'))) \right) - d(\gamma(a(\tau)),\gamma(b(\tau))) \right) \geq L2^{-100k-113},
$$
we are done.
Hence we may assume that 
\begin{equation}
\label{newBound}
\left( \left( \sum_{\tau' \in \mathcal{F}_{\tau,3}} d(\gamma(a(\tau')),\gamma(b(\tau'))) \right) - d(\gamma(a(\tau)),\gamma(b(\tau))) \right) < L2^{-100k-113}.
\end{equation}
Write $\mathcal{F}_{\tau,1} = \{ \tau_i \}_{i=1}^M$ arranged in order of the orientation of $\mathbb{T}$.
Set 
$$
P = \bigcup_{i=1}^{M-1} \{ \gamma(b(\tau_i)) \}.
$$
We will prove
\begin{equation}
\label{PBound}
d(P,\{\gamma(a(\tau)),\gamma(b(\tau)) \} ) \geq L2^{-100k-113}.
\end{equation}
(The proof of this is nearly identical to the proof of (18) in \cite{LiSchul}.)
Suppose \eqref{PBound} is not true.
That is, there is some $j$ so that (without loss of generality)
$d(\gamma(b(\tau_j)),\gamma(a(\tau))) < L2^{-100k-113}$.
Say $\xi$ is the sub arc of $\tau$ defined on $[a(\tau),b(\tau_j)]$.
The arc $\xi$ must contain at least one arc in $\mathcal{F}_{\tau,1}$,
so we have $\diam(\xi) \geq L2^{-100(k+1)-10}$.
Thus there is some point $w \in \xi$ so that
$$
\min \{ d(\gamma(a(\tau)),w), d(w,\gamma(b(\tau_j))) \} \geq L2^{-100(k+1)-12}.
$$
Note that $w \in \tilde{\tau}$ for some $\tilde{\tau} \in \mathcal{F}_{\tau,2}$.
Since $\diam(\tilde{\tau}) \leq L2^{-100(k+2)+4}$,
the triangle inequality gives
$$
\min \{ d(\gamma(a(\tilde{\tau})),d(\gamma(a(\tau))), d(\gamma(a(\tilde{\tau})),\gamma(b(\tau_j))) \} 
\geq L2^{-100k-113}.
$$
Therefore, according to the triangle inequality and the negation of \eqref{PBound}, we have
\begin{align*}
\sum_{\tau' \in \mathcal{F}_{\tau,3}} d(\gamma(a(\tau')),\gamma(b(\tau'))) 
& \geq d(\gamma(a(\tau)),d(\gamma(a(\tilde{\tau})))
+d(\gamma(a(\tilde{\tau})),\gamma(b(\tau_j)))
+d(\gamma(b(\tau_j)),\gamma(b(\tau))) \\
&> L2^{-100k-113} 
+d(\gamma(a(\tau)),d(\gamma(b(\tau_j)))
+d(\gamma(b(\tau_j)),\gamma(b(\tau))) \\
&\geq L2^{-100k-113} + d(\gamma(a(\tau)),\gamma(b(\tau))).
\end{align*}
This contradicts \eqref{newBound}.
A similar argument in the case of $\gamma(b(\tau))$ proves \eqref{PBound}.

For any $i \in \{ 1, \dots, M \}$,
we can repeat the above proof of \eqref{PBound}
replacing $\gamma(a(\tau))$ with $\gamma(a(\tau_i))$ 
and $\gamma(b(\tau_j))$ with $\gamma(b(\tau_i))$ to conclude
\begin{equation}
\label{aiBound}
d(\gamma(a(\tau_i)),\gamma(b(\tau_i)) ) \geq L2^{-100k-113}.
\end{equation}
Indeed, if $d(\gamma(a(\tau_i)),\gamma(b(\tau_i)) \} ) < L2^{-100k-113}$,
we set $\xi = \tau_i$ and follow the previous arguments to obtain
\begin{align*}
\sum_{\tau' \in \mathcal{F}_{\tau,3}} &d(\gamma(a(\tau')),\gamma(b(\tau'))) \\
&\geq d(\gamma(a(\tau)),d(\gamma(a(\tau_i)))
+ d(\gamma(a(\tau_i)),d(\gamma(a(\tilde{\tau})))
+d(\gamma(a(\tilde{\tau})),\gamma(b(\tau_i)))
+d(\gamma(b(\tau_i)),\gamma(b(\tau))) \\
&> d(\gamma(a(\tau)),d(\gamma(a(\tau_i)))
+L2^{-100k-113} 
+d(\gamma(a(\tau_i)),d(\gamma(b(\tau_i)))
+d(\gamma(b(\tau_i)),\gamma(b(\tau))) \\
&\hspace{.3in} \geq L2^{-100k-113} + d(\gamma(a(\tau)),\gamma(b(\tau)))
\end{align*}
which again contradicts \eqref{newBound}. This proves \eqref{aiBound}.

Fix $i \in \{2,\dots,M-1\}$.
We will first establish an estimate on the distance from $L_{\gamma(a(\tau))\gamma(b(\tau_i))}$ to $L_{\tau}$.
Combining \eqref{PBound} and \eqref{aiBound} allows us to bound
$$
\min \{ d(\gamma(a(\tau)), \gamma(a(\tau_i))), 
d(\gamma(a(\tau)), \gamma(b(\tau_i))),
d(\gamma(a(\tau_i)), \gamma(b(\tau_i))),
d(\gamma(b(\tau_i)), \gamma(b(\tau))) \} 
$$
from below by $L2^{-100k+4}2^{-117} \geq 2^{-117} \diam (\tau)$.
Therefore, the assumptions of Theorem~\ref{Goal} are satisfied
with $m = 2^{-217}$ and $\rho = \diam(\tau)$
where
$$
a = \gamma(a(\tau)), 
\quad
z = \gamma(a(\tau_i)),
\quad
v = \gamma(b(\tau_i)),
\quad 
\text{and}
\quad w = \gamma(b(\tau)).
$$
Theorem~\ref{Goal} then gives 
\begin{align}
\frac{d(L_{\gamma(a(\tau))\gamma(b(\tau_i))}(t),L_{\tau})^{2r^2}}{\diam(\tau)^{2r^2 - 1}}
&=
\frac{d(L_{av}(t),L_{aw})^{2r^2}}{\rho^{2r^2-1}} \nonumber \\
&\lesssim d(\gamma(a(\tau)), \gamma(a(\tau_i))) + d(\gamma(a(\tau_i)), \gamma(b(\tau_i))) \label{first}\\
&\hspace{.7in} + d(\gamma(b(\tau_i)), \gamma(b(\tau))) -  d(\gamma(a(\tau)), \gamma(b(\tau))) \nonumber
\end{align}
for any $t \in [0,1]$.

We now establish an estimate on the distance from $L_{\tau_i}$ to $L_{\gamma(a(\tau))\gamma(b(\tau_i))}$.
Pairing this with \eqref{first} will give \eqref{NewLem}.
Choose an arc $\hat{\tau} \in \mathcal{F}_{\tau,2}$ so that 
$\hat{\tau}$ is contained in the arc defined on $[a(\tau),a(\tau_i)]$
and $b(\hat{\tau}) \neq a(\tau_i)$.
Such an arc always exists
because, if it did not, then the only arc in $\mathcal{F}_{\tau,2}$ would be the arc defined on $[a(\tau),a(\tau_i)]$,
and this violates the diameter bounds (2) in the definition of a filtration.
We may follow the proof of \eqref{PBound} to conclude that
\begin{equation}
\label{hatBound}
\min \{ d(a(\tau),b(\hat{\tau})) , d(b(\hat{\tau}),a(\tau_i)) \} \geq L2^{-100k-213}.
\end{equation}
\begin{figure}[H]
\centering
\includegraphics[scale = 1]{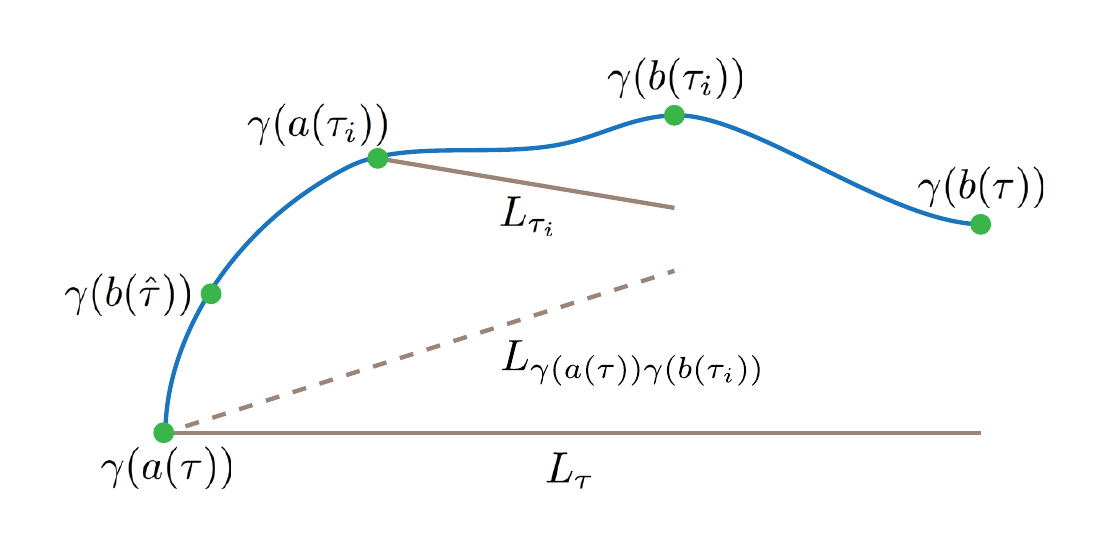}
\caption{}
\label{fig3}
\end{figure}
Indeed, assume (without loss of generality) that $d(a(\tau),b(\hat{\tau})) < L2^{-100k-213}$,
set $\xi$ to be the arc defined on $[a(\tau),b(\hat{\tau})]$,
and note that $\xi$ must contain an arc in $\mathcal{F}_{\tau,2}$.
Thus $\diam(\xi) \geq L2^{-100(k+2)-10}$, and we can choose $\tilde{\tau} \in \mathcal{F}_{\tau,3}$ 
with $\diam(\tilde{\tau}) \leq L2^{-100(k+3)+4}$
so that 
$$
\min \{ d(\gamma(a(\tilde{\tau})),d(\gamma(a(\tau))), d(\gamma(a(\tilde{\tau})),\gamma(b(\hat{\tau}))) \} 
\geq L2^{-100k-213}.
$$
Applying the triangle inequality and the negation of \eqref{PBound} leads to a contradiction of \eqref{newBound} just as before.
This proves \eqref{hatBound}.
We have therefore bounded
$$
\min \{ d(\gamma(a(\tau)), \gamma(b(\hat{\tau}))), 
d(\gamma(a(\tau)), \gamma(a(\tau_i))),
d(\gamma(b(\hat{\tau})), \gamma(a(\tau_i))),
d(\gamma(a(\tau_i)), \gamma(b(\tau_i))) \} 
$$
from below by $2^{-217} \diam (\tau)$ as before.
The assumptions of Theorem~\ref{Goal} are satisfied
with $m = 2^{-217}$ and $\rho = \diam(\tau)$
where
$$
a = \gamma(a(\tau)), 
\quad
z = \gamma(b(\hat{\tau})),
\quad
v = \gamma(a(\tau_i)),
\quad 
\text{and}
\quad w = \gamma(b(\tau_i)).
$$
This gives 
\begin{align}
\frac{d(L_{\tau_i}(t),L_{\gamma(a(\tau)) \gamma(b(\tau_i))})^{2r^2}}{\diam(\tau)^{2r^2 - 1}}
&=
\frac{d(L_{vw}(t),L_{aw})^{2r^2}}{\rho^{2r^2-1}} \nonumber \\
&\lesssim d(\gamma(a(\tau)), \gamma(b(\hat{\tau}))) 
+ d(\gamma(b(\hat{\tau})), \gamma(a(\tau_i))) \label{second} \\
&\hspace{.7in} + d(\gamma(a(\tau_i)), \gamma(b(\tau_i))) 
-  d(\gamma(a(\tau)), \gamma(b(\tau_i))) \nonumber
\end{align}
for any $t \in [0,1]$.

Fix $t \in [0,1]$.
Choose $p \in L_{\gamma(a(\tau)) \gamma(b(\tau_i))}$ so that
$$
d(L_{\tau_i}(t),L_{\gamma(a(\tau)) \gamma(b(\tau_i))}) = d(L_{\tau_i}(t),p)
$$
and $q \in L_{\tau}$ so that
$$
d(p,L_{\tau}) = d(p,q).
$$
Combining \eqref{first} and \eqref{second} gives
\begin{align*}
\frac{d(L_{\tau_i}(t),L_{\tau})^{2r^2}}{\rho^{2r^2-1}} 
&\leq \frac{d(L_{\tau_i}(t),q)^{2r^2}}{\rho^{2r^2-1}} 
\lesssim \frac{d(L_{\tau_i}(t),p)^{2r^2} 
+ d(p,q)^{2r^2}}{\rho^{2r^2-1}}\\
&= \frac{d(L_{\tau_i}(t),L_{\gamma(a(\tau)) \gamma(b(\tau_i))})^{2r^2}}{\rho^{2r^2-1}} 
+ \frac{d(p,L_{\tau})^{2r^2}}{\rho^{2r^2-1}}\\
&\lesssim d(\gamma(a(\tau)), \gamma(a(\tau_i))) + d(\gamma(a(\tau_i)), \gamma(b(\tau_i))) \\
&\hspace{.9in} + d(\gamma(b(\tau_i)), \gamma(b(\tau))) -  d(\gamma(a(\tau)), \gamma(b(\tau))) \\
&\quad + d(\gamma(a(\tau)), \gamma(b(\hat{\tau}))) + d(\gamma(b(\hat{\tau})), \gamma(a(\tau_i))) \\
&\hspace{.9in} + d(\gamma(a(\tau_i)), \gamma(b(\tau_i))) -  d(\gamma(a(\tau)), \gamma(b(\tau_i))) \\
&\leq 2[d(\gamma(a(\tau)), \gamma(b(\hat{\tau}))) + d(\gamma(b(\hat{\tau})), \gamma(a(\tau_i))) + d(\gamma(a(\tau_i)), \gamma(b(\tau_i))) \\
&\hspace{.9in} + d(\gamma(b(\tau_i)), \gamma(b(\tau))) -  d(\gamma(a(\tau)), \gamma(b(\tau)))] \\
&\leq 2 \left( \left( \sum_{\tau' \in \mathcal{F}_{\tau,3}} d(\gamma(a(\tau')),\gamma(b(\tau'))) \right) - d(\gamma(a(\tau)),\gamma(b(\tau))) \right).
\end{align*}

In the case $i=1$ (similarly, $i=M$)
where $a_1 = a$ (similarly, $b_M=b$),
choose $\hat{\tau} \in \mathcal{F}_{\tau,2}$ so that $\hat{\tau}$ is contained in the arc defined on $[a(\tau),a(\tau_2)]$ (similarly, the arc defined on $[a(\tau),a(\tau_M)]$)
and $b(\hat{\tau}) \neq a(\tau_2)$ (similarly, $b(\hat{\tau}) \neq a(\tau_M)$).
We may then apply Theorem~\ref{Goal} with $m = 2^{-217}$ and $\rho = \diam(\tau)$ (and noting that $a(\tau_2) = b(\tau_1)$) to 
$$
a = \gamma(a(\tau)), 
\quad
z = \gamma(b(\hat{\tau})),
\quad
v = \gamma(a(\tau_2)),
\quad 
\text{and}
\quad w = \gamma(b(\tau))
$$
(similarly, $v = \gamma(a(\tau_M))$)
to get, as in the proof of \eqref{first} and \eqref{second},
\begin{align*}
\frac{d(L_{\gamma(a(\tau_1))\gamma(b(\tau_1))}(t),L_{\tau})^{2r^2}}{\diam(\tau)^{2r^2 - 1}}
&\lesssim d(\gamma(a(\tau)), \gamma(b(\hat{\tau}))) + d(\gamma(b(\hat{\tau})), \gamma(b(\tau_1))) \\
&\hspace{.3in} + d(\gamma(b(\tau_1)), \gamma(b(\tau))) -  d(\gamma(a(\tau)), \gamma(b(\tau)))
\end{align*}
for any $t \in [0,1]$.
Similarly, we have the following bound for $i=M$:
\begin{align*}
\frac{d(L_{\gamma(a(\tau_M))\gamma(b(\tau_M))}(t),L_{\tau})^{2r^2}}{\diam(\tau)^{2r^2 - 1}}
&\lesssim d(\gamma(a(\tau)), \gamma(b(\hat{\tau}))) + d(\gamma(b(\hat{\tau})), \gamma(a(\tau_M)))\\
&\hspace{.3in} + d(\gamma(a(\tau_M)), \gamma(b(\tau))) -  d(\gamma(a(\tau)), \gamma(b(\tau))).
\end{align*}
This gives the result. 
\end{proof}

We conclude this subsection with the following estimate:
\begin{proposition}
\label{TSPdiscrete1}
$$
\sum_{\tau \in \mathcal{F}} \beta(\tau)^{2r^2} \diam(\tau) \lesssim \mathcal{H}^1(\Gamma).
$$
\end{proposition}
Once this has been proven, we may argue exactly as in \cite[Corollary 3.3]{LiSchul}
to prove Proposition~\ref{TSPdiscreteG1},
and this completes the subsection.
(It is in that argument that the definition of $\mathcal{G}_1$ is used.)
\begin{proof}[Proof of Proposition \ref{TSPdiscrete1}]
Summing equation \eqref{NewLem} over all $\tau \in \mathcal{F}$ and all $k$ gives
\begin{align*}
\sum_{\tau \in \mathcal{F}} \frac{d_{\tau}^{2r^2}}{\diam(\tau)^{2r^2-1}}
&\leq
C'' \sum_{k=1}^{\infty}  \left( \left( \sum_{\tau \in \mathcal{F}_{k+3}} d(\gamma(a(\tau)),\gamma(b(\tau))) \right) 
- \left( \sum_{\tau \in \mathcal{F}_{k}} d(\gamma(a(\tau)),\gamma(b(\tau))) \right) \right)\\
&\leq 
3C'' \sup_{k \in \mathbb{N}} \sum_{\tau \in \mathcal{F}_{k}} d(\gamma(a(\tau)),\gamma(b(\tau)))
\leq 3C'' \ell(\gamma).
\end{align*}
For any $\tau \in \mathcal{F}$, 
say $\{ \tau_k\}_{k=0}^{\infty}$ is a sequence of subarcs
with $\tau_j \in \mathcal{F}_{\tau,j}$
chosen so that $d_{\tau_j}$ is maximal 
among all subarcs in $\mathcal{F}_{\tau,j}$.
Arguing as in the proof of \cite[Proposition 3.1]{LiSchul}, 
applying \cite[Lemma 3.6]{LiSchul}, 
using Minkowski's integral inequality in $\ell^{2r^2}$, 
and applying property (2) from the definition of filtrations gives
\begin{align*}
\left( \sum_{\tau \in \mathcal{F}} \beta(\tau)^{2r^2}\diam(\tau) \right)^{1/(2r^2)} 
&\leq \sum_{k=0}^{\infty} \left( \sum_{\tau \in \mathcal{F}} \frac{d_{\tau_k}^{2r^2}}{\diam(\tau)^{2r^2-1}} \right)^{1/(2r^2)}\\
&\leq \sum_{k=0}^{\infty} 2^{(-100k+14) \frac{2r^2-1}{2r^2}}
\left( \sum_{\tau \in \mathcal{F}} \frac{d_{\tau_k}^{2r^2}}{\diam(\tau_k)^{2r^2-1}} \right)^{1/(2r^2)}\\
&\leq \left(  3C'' \ell(\gamma) \right)^{1/(2r^2)}
\sum_{k=0}^{\infty} 2^{(-100k+14) \frac{2r^2-1}{2r^2}}\\
&\lesssim \ell(\gamma)^{1/(2r^2)}
\end{align*}
\end{proof}

\subsection{Flat balls}

In this section, we will prove the other half of \eqref{TSPdiscrete}:
\begin{proposition}
\label{TSPdiscreteFlat}
$$
\sum_{B \in \mathcal{G}_2} \beta_{\Gamma}(B)^2 \diam (B) \lesssim \mathcal{H}^1(\Gamma).
$$
\end{proposition}
To do so, we will follow the proof in Section~4 of \cite{LiSchul}
of a similar bound in the Heisenberg group.
As stated at the beginning of that section,
most of the arguments therein may be applied in any general metric space.
The only Heisenberg-specific ingredients of the proof are Lemma~4.1 and equations (23) and (24).
Therefore, in order to prove Proposition~\ref{TSPdiscreteFlat},
it suffices to verify these three facts in $\mathbb{G}$.

Equation (23) in \cite{LiSchul} requires 
$$
\diam(B(p,\lambda r)) \leq \lambda \diam (B(p,r))
$$
for any $r>0$, $p \in \mathbb{G}$, and $\lambda > 1$.
In $\mathbb{G}$, we have
\begin{equation}
\label{23}
\diam(B(p,\lambda r)) = \lambda \diam (B(p,r))
\end{equation}
for any $r>0$, $p \in \mathbb{G}$, and $\lambda > 0$.
This follows from the fact that $\diam(B(p,r)) = 2 r$
for any left invariant, homogeneous metric in $\mathbb{G}$
\cite[Proposition 2.4]{FraSerSerOnthe}.
Moreover, equation (24) in \cite{LiSchul} is a result
of %\footnote{In fact, I think we only need that $\diam (L) = d(L(0),L(1))$, and this is true since the metric is left invariant, since any horizontal line starting at the origin is a Euclidean segment, and since Carnot distances along such segments equal $1/\eta$ times the Euclidean distance.}
\begin{equation}
\label{24}
d(L(t_1),L(t_2)) = |t_1-t_2| \Vert \tilde{\pi}(p^{-1}q) \Vert
\end{equation}
for any horizontal segment $L:[0,1] \to \mathbb{G}$ and any $t_1,t_2 \in [0,1]$.

It remains to prove Lemma~4.1 from \cite{LiSchul}
in the Carnot group setting.
We first establish the following:

\begin{lemma}
\label{tangent}
There is a radius $0< r_0 \leq \tfrac12$ such that,
for any horizontal segment $L$ which intersects $B(0,r_0)$ non-trivially,
$L$ is never tangent to the unit sphere $\partial B(0,1) = \partial B_{\mathbb{R}^N}(0,\eta)$.
\end{lemma}

\begin{proof}
Note that we need only consider those horizontal segments
in $B_{\mathbb{R}^n}(0,2\eta) \times \mathbb{R}^{N-n}$. 
Indeed, the 
projection of the segment to $\mathbb{R}^n \times \{0\}$ is a Euclidean segment traversed at constant speed,
and the restriction of a horizontal segment to a subinterval is still a horizontal segment. %\footnote{This follows from $p*\delta_{t+s}(\tilde{\pi}(p^{-1}q)) = [p*\delta_t(\tilde{\pi}(p^{-1}q))]*\delta_s(\tilde{\pi}(p^{-1}q))$ (which is true since $\delta_t(\tilde{\pi}(p^{-1}q))$ and $\delta_s(\tilde{\pi}(p^{-1}q))$ are colinear in $\mathbb{R}^n \times \{0\}$).}
Hence a horizontal segment will intersect both $B(0,r_0)$ and $\partial B(0,1)$
if and only if 
its restriction to $B_{\mathbb{R}^n}(0,2\eta) \times \mathbb{R}^{N-n}$
(which is also a connected, horizontal segment)
does as well.

Suppose by way of contradiction
that there is a sequence of horizontal segments $L_j:[0,1] \to \mathbb{G}$ 
in $B_{\mathbb{R}^n}(0,2\eta) \times \mathbb{R}^{N-n}$
which intersect $B(0,1/j)$ non-trivially and lie tangent to $\partial B(0,1)$.
Say $L_j = L_{p_jq_j}$ for some $p_j,q_j \in \mathbb{G}$.
These horizontal segments are all 4-Lipschitz since 
$$
d(L_j(s),L_j(t)) 
= d(\delta_s(\tilde{\pi}(p_j^{-1}q_j)),\delta_t(\tilde{\pi}(p_j^{-1}q_j)))
= \tfrac{1}{\eta}|s-t||(q_j)_1 - (p_j)_1|
\leq 4|s-t|.
$$
In particular, 
since each segment $L_j$ meets $B(0,1/j)$,
there is some $M_0>0$ so that $|p_j| < M_0$ for every $j \in \mathbb{N}$.
Write $p_j = (p_j^1,p_j^2)$ and $q_j = (q_j^1,q_j^2)$.
By definition, we have
$$
L_j(s) = \left(p_j^1 + s(q_j^1 - p_j^1), p_j^2 + P\left(p_j,\left(s(q_j^1-p_j^1),0\right) \right) \right)
\quad
\text{for any }
s \in [0,1], \; j \in \mathbb{N}
$$
for some polynomial $P$ given by the BCH formula.
Therefore, by the uniform boundedness of $|p_j|$ and $|q_j^1|$,
there is some $M_1 > 0$ so that $\left| d^i / d s^i (L_j) \right| < M_1$
for every $i,j \in \mathbb{N}$.
The Arzel\`{a}-Ascoli theorem then gives a 
subsequence of these horizontal segments (also called $\{L_j\}$)
converging uniformly in $\mathbb{R}^N$ (and thus in $\mathbb{G}$) to some $C^{\infty}$ curve $L:[0,1] \to \mathbb{G}$
passing through the origin
so that all derivatives of $L_j$ converge uniformly to the corresponding derivatives of $L$.
Note that $L$ itself must also be a horizontal segment.
Indeed, $p_j = L_j(0) \to p$ for some $p \in \mathbb{G}$,
and $\tilde{\pi}(p_j^{-1}q_j) = p_j^{-1} L_j(1) \to (z,0)$ for some $z \in \mathbb{R}^n$.
Thus, for any $q \in \mathbb{G}$ satisfying $q_1 - p_1 = z$, we have
$$
L(s) 
= \lim_{j \to \infty} L_j(s) 
= \lim_{j \to \infty} p_j * \delta_s(\tilde{\pi}(p_j^{-1}q_j)) 
= p * \delta_s((z,0))
= p * \delta_s(\tilde{\pi}(p^{-1}q))
$$
for every $s \in [0,1]$.
Since $L$ is a horizontal segment passing through the origin,
it must be the case that $L$ is a Euclidean line segment in $\mathbb{R}^n \times \{0\}$.\footnote{Divide $L$ 
into two segments: the segment ending at 0 and the segment starting at 0. 
Since both of these must also be horizontal, they must be Euclidean segments.}
In particular, $L$ cannot be tangent to the Euclidean sphere $\partial B_{\mathbb{R}^N}(0,\eta)$.
Since the derivatives of the segments $L_j$ converge uniformly to the derivatives of $L$, 
it is impossible that $L_j$ lies tangent to the sphere for every $j$
(as these segments may only intersect the sphere in a neighborhood of $L$).
This is a contradiction and completes the proof.
\end{proof}

The following is the Carnot group version of Lemma~4.1 from \cite{LiSchul}.
Note that, here, we have the constant $r_0$ included in the inequality, while, in \cite{LiSchul}, the constant is 1.
This, however, is not a problem since the constant $r_0$ depends only on $\mathbb{G}$.

\begin{lemma}
Let $\tau$ be a connected subarc.
Then
\begin{equation}
\label{arcseg}
\sup_{x \in L_{\tau}} d(x,\tau) \leq r_0^{-1} \beta(\tau) \diam(\tau).
\end{equation}
In particular, if we write $I_\tau = [a,b]$, we have
\begin{equation}
\label{endpt}
d(L_{\tau}(1), \gamma(b)) \leq r_0^{-1} \beta(\tau) \diam(\tau).
\end{equation}
\end{lemma}

\begin{proof}
Recall that $\beta(\tau) \diam(\tau) = \sup_{p \in \tau} d(p,L_\tau)$. 
By the invariance of the metric under left translation,
we may assume that $0 = \gamma(a) = L_{\tau}(0)$.
We begin by proving \eqref{endpt}.
Choose $t_0 \in [0,1]$ so that $d(\gamma(b),L_{\tau}(t_0)) = d(\gamma(b),L_{\tau}) \leq \beta(\tau) \diam(\tau)$.
Since $L_{\tau}(1) = \tilde{\pi}(\gamma(b))$ and $L_{\tau}(t_0) = \delta_{t_0}(\tilde{\pi}(\gamma(b)))$
are co-linear in $\mathbb{R}^n \times \{0\}$, 
it follows (as in the proof of Lemma~\ref{BoundLem})
that 
$$
d(L_{\tau}(1),L_{\tau}(t_0)) 
= \left \Vert \delta_{t_0}(\tilde{\pi}(\gamma(b)))^{-1} * \tilde{\pi}(\gamma(b)) \right \Vert 
= \left \Vert \tilde{\pi}[\delta_{t_0}(\tilde{\pi}(\gamma(b)))^{-1} * \gamma(b)] \right \Vert 
\leq d( \gamma(b), L_{\tau}(t_0)).
$$
Therefore, we have
$$
d(L_{\tau}(1), \gamma(b)) \leq d(L_{\tau}(1),L_{\tau}(t_0)) + d(L_{\tau}(t_0), \gamma(b)) \leq 2\beta(\tau) \diam(\tau) \leq r_0^{-1} \beta(\tau) \diam(\tau).
$$

In order to prove \eqref{arcseg}, we will first show that the mapping $f:\tau \to L_{\tau}$ defined as 
$$
f(p) = L_{\tau}(t_0) \quad \text{where } t_0 = \sup \{ t \in [0,1] \, : \, d(L_{\tau}(t),p) \leq r_0^{-1} \beta(\tau) \diam(\tau) \}
$$
is continuous.
(Note that $d(p,L_{\tau}) \leq \beta(\tau) \diam(\tau)$ 
for every $p \in \tau$, so $f$ is well defined.)
In order to prove that $f$ is continuous,
it suffices to prove for every $p \in \tau$ 
that $L_{\tau}$ does not lie tangent to the sphere centered at $p$ with radius $r_0^{-1} \beta(\tau) \diam(\tau)$. 
%Indeed, this will ensure that small changes in the location of $p$ result in small changes in the location of $f(p)$.

Fix $p \in \tau$.
We may translate by $p^{-1}$ and dilate by $r_0 (\beta(\tau) \diam(\tau))^{-1}$
to reduce to the following problem:
show that the horizontal segment $L = \delta_{r_0 (\beta(\tau) \diam(\tau))^{-1}}(p^{-1}*L_{\tau})$
is never tangent to the sphere $\partial B(0,1)$.
This follows from Lemma~\ref{tangent} since
$$
d(0,L) 
= d(0,\delta_{r_0 (\beta(\tau) \diam(\tau))^{-1}}(p^{-1}*L_{\tau}))
= r_0 (\beta(\tau) \diam(\tau))^{-1} d(p,L_{\tau})
\leq r_0
$$
implies that the segment $L$ intersects the ball $B(0,r_0)$ non-trivially.
Therefore, $f$ is continuous. 

Since $\tau$ is connected and $f(\gamma(b)) = L_{\tau}(1)$ by \eqref{endpt}, 
the continuous map $f$ sends $\tau$ onto the interval $[f(0),L_{\tau}(1)] \subset L_\tau$.
Say $t_1 \in [0,1]$ is such that $L_{\tau}(t_1) = f(0)$.
Then, for any $t \in [t_1,1]$,
we have 
$d(L_{\tau}(t),\tau) \leq d(L_{\tau}(t),p) \leq r_0^{-1} \beta(\tau) \diam(\tau)$
for some $p \in \tau$ by the surjectivity of $f$,
and,
for any $t \in [0,t_1]$,
we have 
$$
d(L_{\tau}(t),0) 
= t \Vert \tilde{\pi}(\gamma(b)) \Vert
\leq t_1 \Vert \tilde{\pi}(\gamma(b)) \Vert
= d(L_{\tau}(t_1),0) 
\leq r_0^{-1} \beta(\tau) \diam(\tau).
$$
This proves the lemma.
\end{proof}

With the above lemmas established,
we may now argue exactly as in Section 4 of \cite{LiSchul}
(with the constants therein adjusted appropriately to account for $r_0$)
to conclude Proposition~\ref{TSPdiscreteFlat}.
This, together with Proposition~\ref{TSPdiscreteG1}, finishes the proof of \eqref{TSPdiscrete},
and thus the proof of Theorem~\ref{TSP} is complete.

\section{Step 2 groups}
\label{Step2Sec}

In this section, we will prove Theorem~\ref{TSP2}.
In Theorem~\ref{t:H-TSP1} (proven in \cite{LiSchul}), 
the Traveling Salesman Theorem is established in the Heisenberg group, 
and the exponent on the $\beta$-numbers is 4.
However, this is not the same exponent provided by Theorem~\ref{TSP}.
Indeed, the Heisenberg group has step $r=2$, 
and we have proven the TST in step 2 groups 
where the exponent on the $\beta$-numbers is $2r^2 = 8$.

The following lemma will replace Lemma~\ref{BoundLem}
and will allow us to replace all instances of $2r^2$ with $2r$ in Theorem~\ref{Goal} and in all of the arguments that follow.
This will prove Theorem~\ref{TSP2} and provide a true generalization of Theorem~\ref{t:H-TSP1}.
In the following proof, we will work with $d_{\infty}$ defined on a step 2 group $\mathbb{G}$ as
$$
d_{\infty}(x,y) = N_{\infty}(y^{-1}x) \quad \text{where } N_{\infty}(p) = \max\{ |p_1|,|p_2|^{1/2} \}
$$
for any $p=(p_1,p_2) \in \mathbb{G}$.
Though $d_{\infty}$ is not a true metric
(since a scaling constant is present in the triangle inequality),
it is homogeneous and hence bi-Lipschitz equivalent to the HS-distance $d$ in the sense of \eqref{quasiconvexity}.
This will suffice.
\begin{lemma}
\label{BoundLem2}
Suppose $\mathbb{G}$ is a step 2 Carnot group
and the following bounds hold for some $C=C(\mathbb{G})>0$:
\begin{align}
d(L_v(1),L_w)^{4} = d(\tilde{\pi}(v),L_w)^{4} &\leq C\rho^{3} \Delta, \label{i2} \\
d(L_{vw}(0),L_w)^{4} = d(v,L_w)^{4} &\leq C\rho^{3} \Delta, \label{ii2} \\
d(L_{vw}(1),L_w)^{4} = d(v \, \tilde{\pi}( v^{-1} w) ,L_w)^{4} &\leq C\rho^{3} \Delta. \label{iii2}
\end{align}
Then 
$$
\sup_{t \in [0,1]} d(L_v(t),L_w)^{4} + \sup_{t \in [0,1]} d(L_{vw}(t),L_w)^{4}
\lesssim \rho^{3} \Delta.
$$
\end{lemma}
The hypothesis of this lemma is the same as in Lemma~\ref{BoundLem}
with the substitution $r=2$.
However, the conclusion is different:
we have the exponent $4=2r$ rather than $8=2r^2$.
Once this lemma has been proven, 
the rest of the arguments in the paper follow in exactly the same manner
with all instances of $2r^2$ replaced with 4.
\begin{proof}
As in the proof of Lemma~\ref{BoundLem}, 
consider the horizontal segment $f = L_{vw}$
and the sub-segment %$g = \hat{g}(t(t_2 - t_1) + t_1)$
$$
g(t) = \delta_{t(t_2-t_1) + t_1}(\tilde{\pi}(w))
=\delta_{t_1}(\tilde{\pi}(w)) * \delta_t( \delta_{t_1}(\tilde{\pi}(w))^{-1} \delta_{t_2}(\tilde{\pi}(w)) )
$$
of $L_w$
where $t_1,t_2 \in [0,1]$ are chosen so that
\begin{equation}
\label{close}
d_{\infty}(f(0),g(0))^4 \leq L^4C\rho^{3} \Delta
\quad
\text{and}
\quad
d_{\infty}(f(1),g(1))^4 \leq L^4C\rho^{3} \Delta.
\end{equation}
%Using notation as before, we may write $$f(0) = v,\quad f(1) = v \tilde{\pi}(v^{-1}w),\quad g(0) = \delta_{t_1}(\tilde{\pi}(w)),\quad g(1) = \delta_{t_2}(\tilde{\pi}(w)).$$
Since the BCH formula reduces to $X + Y + \frac12 [X,Y]$ in a step 2 Carnot group, 
we have for any $t \in [0,1]$ and $t' = t(t_2-t_1) + t_1$
\begin{align*}
g(t)^{-1}f(t) 
&=[\delta_{t_1}(\tilde{\pi}(w)) * \delta_t( \delta_{t_1}(\tilde{\pi}(w))^{-1} \delta_{t_2}(\tilde{\pi}(w)) )]^{-1}  v \delta_t(\tilde{\pi}(v^{-1}w)) \\
&=\left(-t'w_1,0)*(v_1 + t(w_1-v_1), v_2 + \tfrac12 [v_1,t(w_1-v_1)] \right) \\
&=\left(v_1 + t(w_1-v_1) - t'w_1, v_2 + \tfrac{t}{2} [v_1,w_1] + \tfrac12 [-t'w_1,v_1 + t(w_1-v_1)]\right) \\
%&=\left(v_1 + t(w_1-v_1) - t'w_1, v_2 + \tfrac{t}{2} [v_1,w_1] - \tfrac12 [-t'w_1,v_1] + \tfrac12[-t'w_1, t(w_1-v_1)] \right) \\
%&=\left(v_1 + t(w_1-v_1) - t'w_1, v_2 + \tfrac{t}{2} [v_1,w_1] - \tfrac{t'}{2} [w_1,v_1] - \tfrac{tt'}{2}[w_1, w_1-v_1] \right) \\
%&=\left(v_1 + t(w_1-v_1) - t'w_1, v_2 + \tfrac{t}{2} [v_1,w_1] + \tfrac{t'}{2} [v_1,w_1] - \tfrac{tt'}{2}[v_1, w_1] \right) \\
&=\left(v_1 + t(w_1-v_1) - t'w_1, v_2 + \tfrac12(t+t'-tt')[v_1,w_1] \right).
\end{align*}
In particular, we have
$$
g(0)^{-1}f(0) = \left(v_1 - t_1w_1,  v_2 + \tfrac{t_1}{2}[v_1,w_1] \right)
\text{ and }
g(1)^{-1}f(1) = \left((1 - t_2)w_1,  v_2+ \tfrac12[v_1,w_1] \right),
$$
so \eqref{close} gives
$$
\max \{ |v_1 - t_1w_1|, |(1 - t_2)w_1|, |v_2 + \tfrac{t_1}{2}[v_1,w_1]|^{1/2}, |v_2+ \tfrac12[v_1,w_1]|^{1/2} \} \lesssim \rho^{3/4} \Delta^{1/4}.
$$

We will now show that $d_{\infty}(f(t),g(t)) = N_{\infty}(g(t)^{-1}f(t)) \lesssim \rho^{3/4} \Delta^{1/4}$ for any $t\in[0,1]$.
Indeed, we first have
$$
|v_1 + t(w_1-v_1) - t'w_1| = |(1-t)(v_1-t_1w_1) + t(1-t_2)w_1| \lesssim \rho^{3/4} \Delta^{1/4}.
$$
To bound the second coordinate, we choose $t_0$ so that $t+t'-tt' = t_0 +  (1-t_0)t_1$.
That is,
$$
t_0 := \frac{t+t'-tt' -t_1}{1-t_1} = t + t(1-t)\left( \frac{t_2-t_1}{1-t_1} \right) \in [0,1].
$$
Therefore,
$$
\left|v_2 + \tfrac12(t+t'-tt')[v_1,w_1]\right| 
\leq t_0\left|v_2 + \tfrac12[v_1,w_1]\right|
 + (1-t_0)\left|v_2 + \tfrac{t_1}{2}[v_1,w_1]\right|
\lesssim \rho^{3/2} \Delta^{1/2}.
$$
Hence $d(L_{vw}(t),L_w)^4 \lesssim d_{\infty}(L_{vw}(t),g(t))^4 \lesssim \rho^3 \Delta$.
In the case $f = L_v$, 
it is similar and simpler to establish the bound on $d(L_v(t),L_w)$.
This completes the proof of the lemma.
\end{proof}

\section{Singular integrals on 1-regular curves}
\label{SIOSec}

%Let $\G$ be Carnot group of homogeneous dimension $Q \geq 3$  identified with $\mathbb{R}^N$ for some $N \in \mathbb{N}$ via the exponential coordinates on $\mathfrak{g}$.  
Recall that if $(X,d)$ is a metric space, an $\ha^{1}$-measurable set $E \subset X$ is \emph{
$1$-(Ahlfors)-regular}, if there exists a constant $1\leq C
<\infty$, such that
\begin{displaymath}
C^{-1} r \leq \mathcal{H}^1(B(x,r)\cap E) \leq C r
\end{displaymath}
for all  $x\in E$, and $0<r\leq \mathrm{diam}(E)$. In this section we are going to prove Theorem \ref{siosintro}, which we reformulate in a more precise manner below.
\begin{theorem}
\label{sios}
Let $(\G,d)$ be Carnot group of step $r\geq 2$ equipped with a homogeneous metric $d$. Let $K_d : \G \stm \{0\} \ra [0,\infty)$ be defined by
$$K_d(p)=\frac{d(NH(p),0)^{2r^3}}{d(p,0)^{2r^3+1}},$$
and let $E$ be a $1$-regular set which is contained in a $1$-regular curve.  Then the corresponding truncated singular integrals
$$T^\ve f\,(p)=\int_{E \stm B_{d}(p,\ve)} K_d(q^{-1} \cdot p) f(q) \,d \ha^1(q)$$
are uniformly bounded in $L^2(\ha^1 |_E)$.
\end{theorem}
\begin{proof}
The proof follows as in the proof of \cite[Theorem 1.3]{ChoLi} once we have at our disposal Theorem \ref{TSP} and Lemma \ref{kernellemma}. Nevertheless we will provide an outline for the convenience of the reader. To simplify notation we let $\mu = \ha^1|_E$ and $K=K_d$. Since $E$ is $1$-regular there exists some constant $c_{\mu} \in(0,1]$ such that \begin{align*}
  c_{\mu} r \leq \mu(B(x,r)) \leq c_{\mu}^{-1} r, \qquad \forall x \in E, r > 0.
\end{align*}
We first observe that the kernel $K$ is a symmetric $1$-dimensional {\it Calder\'on-Zygmund} (CZ)-kernel, see \cite[Definition 2.6 and Lemma 2.7]{ChoLi}. 
We will use the $T1$-theorem (which we explain more in the following) to prove that the operators $T^\ve$ are uniformly bounded on $L^2(\mu)$. For this reason, we need a system of dyadic-like cubes associated to the set $E$. These systems were introduced by David in \cite{david-wavelets} for regular Euclidian sets and later generalized by Christ \cite{Chr} to any regular set of a geometrically doubling metric space. In particular for the set $E$, there is a constant $c_{d} \in (0,1]$ and a family of partitions $\Delta_{j}$ of $E$, $j \in \Z$, with the following properties;
\begin{itemize}
\item[(D1)] If $k \leq j$, $Q \in \Delta_{j}$ and $Q' \in
\Delta_{k}$, then either $Q \cap Q' = \emptyset$, or $Q \subset
Q'$. 
\item[(D2)] If $Q \in \Delta_{j}$, then $\diam Q \leq 2^{-j}$.
\item[(D3)] Every set $Q \in \Delta_{j}$ contains a set of the form
$B(p_{Q},c_{d}2^{-j}) \cap E$ for some $p_{Q} \in Q$.
\end{itemize}
We will call the sets in $\Delta := \cup \Delta_{j}$ the \emph{dyadic cubes} of $E$.  %For $Q \in \Delta_{j}$, we define $\ell(Q) :=%2^{j}$. 
%Thus, by (ii), we have $\diam (Q) \leq \ell(Q)$ for $Q \in \Delta$. 
For a cube $S \in \Delta$, we define
\begin{align*}
  \Delta(S) := \{Q \in \Delta : Q \subseteq S\}.
\end{align*}
Given a cube $Q \in \Delta$ and $\lambda \geq 1$, we define
\begin{align*}
  \lambda Q := \{x \in E : d(x,Q) \leq (\lambda - 1) \diam(Q)\}.
\end{align*}
It follows from (D2), (D3) and the $1$-regularity of $E$ that if $Q \in \Delta_j$,
$$c_d 2^{-j}\leq \diam(Q) \leq 2^{-j} \mbox { and }c_d c_\mu 2^{-j} \leq \mu(Q) \leq c_\mu^{-1} 2^{-j}.$$

To prove the $L^{2}(\mu)$ boundedness of the operator $T^{\ve}$  it suffices to verify that there exists a uniform bound $C < \infty$ that can depend on $c_{\mu}, c_d$ so that
\begin{align}
  \|T^\ve \chi_S\|_{L^2(S)}^2 \leq C \mu(S), \qquad \forall S \in \Delta, \forall \ve > 0 \label{e:T_eps-bound}
\end{align}
where $L^2(S):=L^2(\mu|_S)$. These conditions suffice by the $T1$ theorem of David and Journ\'e, applied in the homogeneous metric measure space $(E,d,\mu)$, see \cite[Theorem 3.21]{tolsabook}. Notice that since $K$ is symmetric, $(T^{\ve})^{\ast}=T^{\ve}$ where $(T^{\ve})^{\ast}$ is the formal adjoint of $T^{\ve}$, see also \cite[Remark 2.6]{CFOsios}. The statement in Tolsa's book is formulated for Euclidean spaces, but the proof works with minor standard changes in homogeneous metric measure spaces; the details can be found in the honors thesis of Fernando \cite{Fer}. Observe that we may suppose that $E$ is a 1-regular rectifiable curve as taking a subset can only decrease the $L^2(\mu)$-bound of $T^\ve \chi_S$.

We will now decompose our singular integral dyadically. This approach was used in \cite{CFOsios} and \cite{ChoLi} and is inspired by \cite{tolsaplms}. Let $\psi : \R \to \R^+$ be a Lipschitz function so that $\chi_{B(0,1/2)} \leq \psi \leq \chi_{B(0,2)}$.  For any $j \in \Z$ we let $\psi_j : \G \to \R$  such that $\psi_j(z)=\psi(2^j d(z,0))$ and we set $\phi_j := \psi_j - \psi_{j+1}$.  Note that $\phi_j$ is supported on the annulus $B(0,2^{1-j}) \backslash B(0,2^{-2-j})$ and for any $N \in \Z$, $\sum_{n \leq N}=1-\psi_{N+1}$, hence
\begin{align}
  \chi_{\G \backslash B(0,2^{-N})} \leq \sum_{n \leq N} \phi_n \leq \chi_{\G \backslash B(0,2^{-N-2})}. \label{e:partial-phi-sum}
\end{align}
For each $j \in \Z$, we let $K_{(j)} := \phi_j \cdot K$ and we define 
\begin{align*}
  T_{(j)}f(x) = \int  K_{(j)}(y^{-1}x)f(y) ~d\mu(y).
\end{align*}
for nonnegative functions $f \in L^2(\mu)$. For $N \in \Z$  let $S_N = \sum_{n \leq N} T_{(n)}$.  As the kernel $K$ is positive,  \eqref{e:partial-phi-sum} implies the following {\it pointwise} estimates for any nonnegative function $f \in L^2(\mu)$ from
\begin{align*}
  0 \leq T_1^\ve f \leq S_{n} f, \qquad \forall \ve \geq 2^{-n}.
\end{align*}
Thus, to establish the uniform bound \eqref{e:T_eps-bound}, it suffices to show that there exists some absolute constant $C < \infty$ such that
\begin{equation}
\label{snbound}
  \|S_n \chi_S\|_{L^2(S)}^2 \leq C \mu(S), \qquad \forall S \in \Delta, \forall n \in \Z.
\end{equation}
We now fix $S \in \Delta_\ell$ for some $\ell \in \Z$.  We will show that for any $j \in \Z$ and $x \in E$, we have
  \begin{equation}
\label{e:Tj-beta}
    T_{(j)}1(x) \lesssim_{c_{\mu}} \beta_E(x,2^{1-j})^{2r^2}. 
  \end{equation}
In order to prove \eqref{e:Tj-beta} we need the following lemma which was first proven in the case of the Heisenberg group in \cite[Lemma 3.3]{LiSchul2}.
\begin{lemma}
\label{kernellemma}
Let $(\G,d)$ be Carnot group of step $r\geq 2$ equipped with a homogeneous metric $d$. Then
\begin{equation}
\label{kerneleq}
\frac{d(NH(a^{-1}b),0)^{r}}{d(a,b)^{r-1}} 
\lesssim \max \{ d(a,L),d(b,L) \}
\end{equation}
for any $a,b \in \mathbb{G}$ and any horizontal line $L \subset \mathbb{G}$.
\end{lemma}

\begin{proof}
For any $p \in \mathbb{G}$, 
we will write $p = (p_1,\dots,p_r)$
where $p_k \in \mathbb{R}^{v_k}$ and $v_k = \dim V_k$.
As in the previous section, we will utilize the homogeneous norm
$$
\|p\|_{\infty} = \max\{ |p_k|^{1/k} \}_{k=1}^r.
$$
For $x,y \in \G$ we will denote $d_\infty(x,y):=  \|y^{-1}x \|_\infty$. 
Note that $d_{\infty}$ is not a true metric since it does not satisfy the triangle inequality.
Rather, there is a sub-additive constant $C_\infty \geq 1$.
Regardless, it follows that $d_\infty$ is globally equivalent to $d$ 
in the sense of \eqref{quasiconvexity}.
Fix $a,b \in \mathbb{G}$ and a horizontal line $L \subset \mathbb{G}$.
Note that
\begin{equation}
\label{NHcompactbound}
\Vert NH(a^{-1}b) \Vert_\infty
%= \Vert \tilde{\pi}(a^{-1}b) a^{-1}b \Vert_\infty
\leq C_{\infty} (\Vert \tilde{\pi}(a^{-1}b) \Vert_\infty + \Vert a^{-1}b \Vert_\infty)
\leq 2 C_{\infty} \Vert a^{-1} b \Vert_\infty = 2C_{\infty}d_\infty(a,b).
\end{equation}

If $d_\infty(a,b) < \max \{ d_\infty(a,L),d_\infty(b,L) \}$,
then 
$$
\frac{d_\infty(NH(a^{-1}b),0)^{r}}{d_\infty(a,b)^{r-1}} 
\leq (2C_{\infty})^r d_\infty(a,b)
< (2C_{\infty})^r  \max \{ d_\infty(a,L),d_\infty(b,L) \}.
$$
Thus we may assume $d_\infty(a,b) \geq \max \{ d_\infty(a,L),d_\infty(b,L) \}$.

Write $d := d_\infty(a,b)$,
and choose $\ell_a,\ell_b \in L$ so that
$
d_\infty(a,L) = d_\infty(a,\ell_a)
$
and
$
d_\infty(b,L) = d_\infty(b,\ell_b)
$.
Without loss of generality,
we may assume that $\ell_a = 0$
so that
$\ell_b = (x,0,\dots,0)$.
We have
$$
\frac{\Vert NH(a^{-1}b) \Vert_\infty^{r}}{d^{r-1}} 
= d \left( \frac{\Vert NH(a^{-1}b) \Vert_\infty}{d} \right)^r
= d \Vert NH(\delta_{1/d}(a^{-1}b)) \Vert_\infty^r
\lesssim d |NH(\delta_{1/d}(a^{-1}b))|.
$$
This last inequality follows from %the equivalence of $d$ and $d_{\infty}$ and from 
\eqref{compact}
with a constant depending only on $\G$ 
since \eqref{NHcompactbound} implies $NH(\delta_{1/d}(a^{-1}b)) \in B_{\infty}(0,2C_{\infty})$
for any choice of $a$ and $b$.
We can write $c = \ell_b^{-1} b$ so that
$
a^{-1}b = a^{-1} \ell_b c
$
and 
$
\Vert c \Vert_\infty = d_\infty(b,L)
$.
This gives
$$
NH(\delta_{1/d}(a^{-1}b))
= \tilde{\pi}(\delta_{1/d}(a^{-1} \ell_b c))^{-1} * \delta_{1/d}(a^{-1} \ell_b c)
= (0,Q)
$$
where $Q$ is a Lie bracket polynomial 
determined by the BCH formula.
As in the proof of Lemma~\ref{NH},
$Q$ is a finite sum of constant multiples of terms of the form
\begin{equation}
\label{brackets3}
[Z_1,[Z_2,\cdots,[Z_{k-2},[Z_{k-1},Z_k]]\cdots]]
\end{equation}
where each $Z_i$ is either $a_i/d^i$, $c_i/d^i$, or $x/d$.
(Again, we are abusing notation
and identifying each $a_i$ and $c_i$ with the associated vector in $V_i \subset \mathfrak{g}$.)
The definition of $\| \cdot \|_{\infty}$ gives
$$
\frac{|a_i|}{d^i}
\leq \left( \frac{d_\infty(a,\ell_a)}{d_\infty(a,b)} \right)^i
=\left( \frac{d_\infty(a,L)}{d_\infty(a,b)} \right)^i
\leq \frac{d_\infty(a,L)}{d_\infty(a,b)} 
$$
since, by assumption, $d_\infty(a,L) \leq d_\infty(a,b)$.
Similarly, $|c_i|/d^i \leq d_\infty(b,L) / d_\infty(a,b)$.
We also have
\begin{equation*}
\begin{split}
C_{\infty}^{-2} |x|  = C_{\infty}^{-2} \Vert \ell_b \Vert_\infty 
\leq \Vert a \Vert_\infty + \Vert a^{-1} b \Vert_\infty + \Vert b^{-1} \ell_b \Vert_\infty = d_\infty(a,L) + d_\infty(a,b) + d_\infty(b,L)
\leq 3 d.\end{split}\end{equation*}
Therefore, $|x|/d \leq 3C_{\infty}^{2}$.
Now, each nested bracket of the form \eqref{brackets3}
must contain at least one term $a_i/d^i$ or $c_i/d^i$
(since, otherwise, we would have $Z_i = x/d$ for each $i$, so the brackets would all vanish).
Since $\max \{ d_\infty(a,L),d_\infty(b,L) \} / d_\infty(a,b) \leq 1$, this gives
$$
\left|[Z_1,\cdots,[Z_{k-1},Z_k]]\cdots] \right|
\leq \prod_{i=1}^k |Z_i|
\lesssim \frac{\max \{ d_\infty(a,L),d_\infty(b,L) \}}{d_\infty(a,b)}.
$$
Since the sum in the BCH formula is finite, we have
$$
\frac{d_\infty(NH(a^{-1}b),0)^{r}}{d_\infty(a,b)^{r-1}} 
\lesssim d_\infty(a,b) |NH(\delta_{1/d}(a^{-1}b))|
\lesssim \max \{ d_\infty(a,L),d_\infty(b,L) \}.
$$
This completes the proof of the lemma.
\end{proof}

Let $A = E \cap A(x,2^{-2-j},2^{1-j})$.  Since $\psi$ is Lipschitz, we have $\phi_j(y^{-1}x) \lesssim 2^{j+2} d(y,x)$. 
Hence
  \begin{equation*}
  \begin{split}
    T_{(j)}1(x) = &\int_E \phi_j(y^{-1}x) K(y^{-1}x) ~d\mu(y) \lesssim 2^{j+2} \int_A \frac{d(NH(y^{-1}x), 0)^{2r^3}}{d(y,x)^{2r^3}} ~d\mu(y)  \\
    &\lesssim \sup_{y \in A} \frac{d(NH(y^{-1}x), 0)^{2r^3}}{d(x,y)^{2r^3}}.
    \end{split}
  \end{equation*}
Observe that, if $y \in A$, it holds that $d(x,y) \geq 2^{-j-2}$. Moreover, there exists a horizontal line $L$ such that  \begin{align*}
    \beta_{\{x,y\}}(x,2^{1-j}) = \frac{\max\{d(x,L),d(y,L)\}}{2^{1-j}} \gtrsim \frac{\max\{d(x,L),d(y,L)\}}{d(x,y)} \overset{\eqref{kerneleq}}{\gtrsim} \frac{d(NH(x^{-1}y),0)^{r}}{d(x,y)^r}.
  \end{align*}
Hence \eqref{e:Tj-beta} follows as $\beta_E(B(x,2^{1-j})) \geq \beta_{\{x,y\}}(B(x,2^{1-j}))$. For any $Q \in \Delta$ we define $$\beta_E(Q):=\beta_E (p_Q, 2c_d^{-1}+1).$$ Note that if $R \in \Delta_j$ for some $j \in \Z$ then  \eqref{e:Tj-beta} implies that for any $\alpha > 0$
 \begin{align}
 \int_R T_{(j)}1(x)^\alpha ~d\mu(x) \lesssim_{c_\mu} \beta_E(R)^{2r^2\alpha} \mu(R). \label{e:Tj-int-beta}
 \end{align}
 
 Using \eqref{e:Tj-beta} and \eqref{e:Tj-int-beta} and arguing exactly as in \cite[pp 1416-1417]{ChoLi} we deduce that 
\begin{equation}
\label{snbound2}
\|S_n\chi_S\|_{L^2(S)}^2 \lesssim_{c_\mu} \sum_{Q \in \Delta(S^*)} \beta(Q)^{2r^2}\mu(Q) 
\end{equation}
where $S^*$ is the unique cube in $\Delta_{\ell-2}$ such that $S \subset S^*$. 
Using Theorem \ref{TSP} it is not difficult to show (see e.g. the discussion in \cite[Proposition 3.1]{ChoLi}) that there exists an absolute constant $C:=C(c_\mu)>0$ such that, for any $P \in \Delta$, we have
 \begin{align}
    \sum_{Q \in \Delta(P)} \beta_E(Q)^{2r^2} \mu(Q) \leq C \mu(P). \label{e:tsp}
 \end{align}
Now \eqref{snbound} follows by  \eqref{snbound2}, \eqref{e:tsp} and the $1$-regularity of $\mu$. The proof is complete.
\end{proof}

\bibliographystyle{acm}
\bibliography{zimbib}

\end{document}